\documentclass[12pt]{amsart}
\usepackage{latexsym, url}
\usepackage{amsthm}
\usepackage{amsmath}
\usepackage{amsfonts}
\usepackage{amssymb}
\usepackage[dvips]{graphicx}
\usepackage{xypic}
\input xy
\xyoption{all}

\newtheorem{theo}{Theorem}[section]
\newtheorem{lemma}[theo]{Lemma}
\newtheorem{assume}[theo]{Assumption}

\newtheorem{nota}[theo]{Notation}
\newtheorem{propo}[theo]{Proposition}
\newtheorem{defi}[theo]{Definition}
\newtheorem{coro}[theo]{Corollary}
\newtheorem{rem}[theo]{Remark}

\newtheorem{exams}[theo]{Examples}
\newtheorem{exam}[theo]{Example}

\newcommand\card{\operatorname{card}}
\newcommand\Inj{\operatorname{Inj}}

\newcommand\Norm{\operatorname{\bf Norm}}
\newcommand\PNorm{\operatorname{\bf PNorm}}
\newcommand\iso{\operatorname{iso}}
\newcommand\Ord{\operatorname{\bf Ord}} 
\newcommand\Gra{\operatorname{\bf Gra}}

\newcommand\op{\operatorname{op}}
\newcommand\cell{\operatorname{cell}}
\newcommand\ap{\operatorname{ap}}

\newcommand\id{\operatorname{id}}

\newcommand\Set{\operatorname{\bf Set}}
\newcommand\Met{\operatorname{\bf Met}}
\newcommand\CMet{\operatorname{\bf CMet}}
\newcommand\PMet{\operatorname{\bf PMet}}
\newcommand\Ban{\operatorname{\bf Ban}}

\newcommand\colim{\operatorname{colim}}

\newcommand\cd{\mathcal {D}}

\newcommand\ch{\mathcal {H}}

\newcommand\ce{\mathcal {E}}

\newcommand\ck{\mathcal {K}}
\newcommand\cl{\mathcal {L}}
\newcommand\cm{\mathcal {M}}

\newcommand\cw{\mathcal {W}}
\newcommand\cx{\mathcal {X}}

\newcommand\eps{\varepsilon}
\newcommand\pa{\parallel}

\date{October 12, 2021}
 
\begin{document}
\title[Approximate injectivity and smallness in metric-enriched categories]
{Approximate injectivity and smallness in metric-enriched categories}
\author[J. Ad\'amek, and J. Rosick\'{y}]
{J. Ad\'amek, and J. Rosick\'{y}}
\thanks{Supported by the Grant Agency of the Czech Republic under the grant 
               19-00902S} 
\address{
\newline J. Ad\'amek\newline
Department of Mathematics,\newline 
Faculty of Electrical Engineering,\newline
Czech Technical University in Prague,\newline
Czech Republic\newline
and Department of Theoretical Computer Science,\newline
Technical University Braunschweig,\newline
Germany}
\email{j.adamek@tu-bs.de}

\address{\newline J. Rosick\'{y}\newline
Department of Mathematics and Statistics,\newline
Masaryk University, Faculty of Sciences,\newline
Kotl\'{a}\v{r}sk\'{a} 2, 611 37 Brno,\newline
 Czech Republic}
\email{rosicky@math.muni.cz}

\begin{abstract}
Properties of categories enriched over the category of metric spaces are investigated and applied to a study of well-known constructions of metric and Banach spaces. We prove e.g. that weighted limits and colimits exist in a metric-enriched category iff
ordinary limits and colimits exist and $\eps$-(co)equalizers are given
by $\eps$-(co)isometries for all $\eps$.

An object is called
approximately injective w.r.t. a morphism $h:A \to A'$ iff
morphisms from $A$ into it are arbitrarily close to those morphisms that
factorize through $h$. We investigate classes of objects specified
by their approximate injectivity w.r.t. given morphisms. They are
called approximate-injectivity classes.  And we also study,
conversely, classes of morphisms specified by the property that
certain objects are approximately injective w.r.t. them.

For every class of morphisms satisfying a mild
smallness condition we prove that the corresponding approximate-injectivity class
is weakly reflective, and we study the properties of the reflection
morphisms. As an application we present a new categorical proof of the
essential uniqueness of the Gurarii space.

\end{abstract} 
\keywords{}
\subjclass{}

\maketitle

\section{Introduction}
Categories enriched over $\Met$, the category of metric spaces and nonexpanding maps, play an important role in various realms, e.g.,
in the study of quantitative algebras \cite{MPP}, in continuous logic \cite{HH} and in a related theory of approximate Fra\"iss\'e
limits \cite{BY}, \cite{K1} and \cite{L}. These papers have led to a general theory of approximate injectivity developed in \cite{RT} and applied e.g. to a categorical proof of the existence of the Gurarii space, see \cite{Gu,GK} (Gurarii space is recalled in \ref{gur}(4) below). 

Recall that injectivity of objects w.r.t a morphism $h:A\to A'$ is one of the fundamental categorical concepts in algebra and topology: 
an object $K$ is called injective if every morphism $f: A\to K$ is equal to $f'\cdot h$ for some $f': A'\to K$. In $\Met$-enriched 
categories, $K$ is called \textit {approximately injective} if every morphism $f:A\to K$ has arbitrarily small distances from the morphisms of the form $f'\cdot h$, see \cite{RT}. 
We present applications of this concept for which we need $\eps$-pushouts introduced in \cite{RT} but previously considered in special cases in \cite{L}, \cite{G}, or \cite{GK}. An $\eps$-pushout of a span of morphisms is a square
universal among squares commuting up to $\eps$, see \ref{D:push}. 
Similarly, we work with $\eps$-coequalizers
and, in general, $\eps$-colimits introduced in \cite{RT}. We show that they are weighted colimits
in the sense of enriched category theory. In fact, we prove that a metric enriched category has weighted limits and colimits
iff it has ordinary ones and $\eps$-equalizers exist and are formed by isometries (introduced in \cite{RT}) and dually. We also show that isometries  lead to an important factorization system on $\Met$-enriched categories. Another useful concept is that of an $\eps$-isometry introduced in \cite{K1} which form the $\eps$-cancellable closure of isometries. Their role for approximate Fra\"iss\'e limits
was established in \cite{K1} and we transfer it to approximate injectivity.

Our main examples of $\Met$-enriched categories are $\Met$ itself, its full subcategory $\CMet$ on complete metric spaces, and the category $\Ban$ of (real or complex) Banach spaces and linear maps of norm $\leq 1$.  

Approximate injectivity classes, i.e. classes of objects approximately injective to a class $\ch$ of morphisms, were studied in \cite{RT}. In the present paper, we develop this theory further, e.g. we investigate classes of morphisms consisting, for a given class $\cx$ of objects, of precisely those morphisms for which all objects of $\cx$ are approximately injective. In case of injectivity, these classes are closed under transfinite composites, and they are stable under pushout and are (left) cancellable. In the approximate injectivity case, pushouts are replaced by $\eps$-pushouts and  cancellability by approximate cancellability.  We also develop an approximate small-object argument corresponding to the small-object argument in case of injectivity. The latter constructs weak reflections in the full subcategory of injective objects, which is important in homotopy theory where it yields e.g. fibrant replacements. Our approximate small-object argument provides weak reflections for the full subcategory of approximately injective objects. For this purpose, we introduce approximately small objects ensuring the convergence of the small-object argument. We prove that finite-dimensional Banach spaces are approximately small w.r.t. isometries and apply it to a new categorical proof of the essential uniqueness of the Gurarii space.

\vskip 1mm
\noindent
{\bf Acknowledgement.} 
We are grateful to W. Kubi\'s for valuable discussions about Banach spaces and 
to I. Di Liberti and J. Velebil for their help with understanding $\eps$-colimits as weighted colimits.

\section{Metric spaces}
We denote by $\Met$ the category of (generalized) metric spaces and nonexpanding maps. A metric is a function from $X\times X$ to 
$[0, \infty]$ (distance $\infty$ is allowed, therefore the proper name would be `generalized metric') satisfying the usual axioms. Given  metric spaces $C$ and $C'$, a function $f\colon C \to C'$ is \textit{nonexpanding} iff for all $x$, $y\in C$ we have 
$$
d(x,y) \geq d\big( f(x), f(y)\big)\,.
$$

In case that for all $x$, $y\in C$ we have
$$
d(x,y) 
= d\big( f(x), f(y)\big)\,,
$$
$f$ is called an \textit{isometry}.

\begin{nota}\label{N:Met}
{
\em
For every real number $\varepsilon\geq 0$ we write
$$
x\sim_\varepsilon y \quad \mbox{instead of}\quad d(x,y)\leq \varepsilon\,.
$$
We use the letter $\varepsilon$ to denote a (variable) real number $\geq 0$.
}
\end{nota}

\begin{rem}\label{R:Met}
\em 
{
(1) The category $\Met$ is symmetric monoidal closed w.r.t.\ $A\otimes B$ having the underlying set $A\times B$ and the metric
$$
d\big((a,b), (a', b')\big) = d(a, a') + d(b, b')\,.
$$
Here $[A,B]$ is the set of all nonexpanding maps with the metric 
$$
d(f,g) = \sup_{\alpha\in A} d\big(f(a), g(a)\big)\,.
$$

(2) Let $\lambda$ be a regular cardinal. Recall that an object $A$ of a category $\ck$ is \textit{$\lambda$-presentable} if its hom-functor $\ck(A,-):\ck\to\Set$ preserves $\lambda$-directed colimits (see \cite{AR}). That is, given a $\lambda$-directed diagram $D$ with colimit cocone $k_i:K_i \to K$ $(i \in I)$, for every morphism $f:A \to K$ 
\begin{enumerate}
\item[(a)] a factorization through $k_i$ exists for some $i\in I$, and

\item[(b)] given two factorizations $f',f'': A \to K_i$ of $f$, we can find a connecting morphism
$k_{i,j}: K_i \to K_j $ of $D$ merging them, i.e., $k_{i,j}f'=k_{i,j}f''$.
\end{enumerate}
A category $\ck$ is \textit{locally $\lambda$-presentable} if it is cocomplete and has a set of $\lambda$-presentable objects whose closure under $\lambda$-directed colimits is all of $\ck$.
Every such category is complete, wellpowered and cowellpowered, see \cite{AR}.
}
\end{rem}

\begin{exams}\label{ex}
{
\em
(1) $\Met$ is locally $\aleph_1$-presentable. Moreover, the forgetful functor $U:\Met\to\Set$ pre\-ser\-ves $\aleph_1$-directed colimits (see \cite{LR} 4.5(3)). In particular, $\Met$ is complete and cocomplete. $U$ has a left adjoint sending a set $X$ to the metric space on $X$ where all distinct points have distance $\infty$. These metric spaces are called \textit{discrete}.

(2) The subcategory $\CMet$ of all complete spaces is a full reflective subcategory of $\Met$ closed under $\aleph_1$-directed colimits. 
Thus it also is locally
$\aleph_1$-presentable, see \cite{AR}, 1.46.  $\CMet$ is also a symmetric monoidal closed category
w.r.t. the structure inherited from $\Met$.

(3) The category $\Ban$ of real (or complex) Banach spaces and linear maps of norm at most $1$ is locally $\aleph_1$-presentable (see \cite{AR} 1.48).

(4) Let $\PMet$ be the category of (generalized) pseudometric spaces and nonexpansive maps. (The  difference is just that for a pseudometric we do not  require $d(x,y)> 0$ if $x\ne y$.) Following \cite{LR} 4.5(3), $\PMet$ is locally $\aleph_1$-presentable.
Moreover, the forgetful functor from $\PMet$ to $\Set$ is topological, see \cite{AHS}, 21.8(1). Indeed, given  pseudometric spaces 
$K_i$, $i\in I$, and a cocone in $\Set$, $f_i\colon K_i \to K$ ($i\in I$), the following  pseudometric $d$ on the set $K$ is  the final one making all $f_i$ nonexpanding:
\[
\tag{$\ast$}d(x,y) =\inf \sum_{k=1}^n r_k\,.
\]
The infimum ranges over $n$-tuples of pairs $u_k$, $v_k\in K_{i_k}$ (for $i_1, \dots , i_n \in I$) of distance $r_k$ such that 
$x= f_{i_1}(u_1)$, $y=f_{i_n}(v_n)$ and for all $k=1, \dots , n-1$ we have $f_{i_k}(v_k) = f_{i_{k+1}} (u_{k+1})$.

The full subcategory $\Met$ of $\PMet$ is reflective: the reflection of a pseudometric space C is its metric quotient
$$q_K:K \to K/\cong 
$$
where $x \cong y$ iff $d(x,y)=0.$
}
\end{exams}
The category $\Met$ is complete and cocomplete and directed colimits are given as follows.

\begin{lemma}\label{filtered}
Let $(k_{ij}:K_i\to K_j)_{i\leq j\in I}$ be a directed diagram in $\Met$ with a colimit $(k_i:K_i\to K)_{i\in I}$. Then 
\begin{enumerate}
\item $k_i$ are jointly surjective: $K=\bigcup\limits_{i\in I} k_i [K_i]$, and
\item for every $i\in I$ and every pair $x$, $y \in K_i$  
$$
d(k_i(x),k_i(y))=\inf_{j\geq i}d(k_{ij}(x),k_{ij}(y)).
$$
\end{enumerate} 
\end{lemma}
\begin{proof}
Let $\bar k_i \colon UK_i\to \bar K$ ($i\in I$) be the colimit of $U(k_{ij})$ in $\Set$. 
Denote by $\bar d$ the final pseudometric on $\bar K$,  see \ref{ex}(4). For this pseudometric space the cocone 
$\bar k_i\colon K_i\to \bar K$ ($i\in I$) is a colimit in $\PMet$. 
If $q\colon \bar K \to K$ denotes the metric quotient of $\bar K$, then the cocone
$$
k_i = q \cdot \bar k_i \colon K_i \to K
$$
is a colimit in $\Met$, see \ref{ex}(4).

Property (1) is clear. To verify (2) since our diagram is directed, $(\ast)$ from \ref{ex}(4) reduces to
$$
\bar d(\bar k_i(x),\bar k_i(y))=\inf_{j\geq i}d(k_{ij}(x),k_{ij}(y)).
$$
Thus (2) follows from
$$
d(k_i(x), k_i(y)) = \bar d(\bar k_i(x), \bar k_i(y)).
$$
\end{proof}  

\begin{rem}\label{directed}
{
\em
(1) For a $\lambda$-directed diagram with $\lambda$ uncountable, condition (2) can be strengthened to
 \begin{enumerate}
  \item[(2')] for every subset $M$ of $K_i$ of power less than $\lambda$ there exists a connecting  map $k_{ij} \colon K_i  \to K_j$ of $D$ such that for all $x,y$  in $M$we have
$$
d \big(k_{ij}(x), k_{ij}(y)\big) = d\big( k_i(x), k_i(y)\big).
$$
\end{enumerate}
(2) For every directed diagram of isometries $k_{ij}$, all $k_i$ are isometries. 

(3) Directed colimits in $\CMet$ are completions of those in $\Met$.

(4) Analogously, directed colimits in $\Ban$ are completions of those in the category $\Norm$ of normed vector spaces and linear maps
of norm $\leq 1$. Those colimits are described analogously to \ref{filtered} with (2) replaced by  
$$
\pa k_i(x)\pa=\inf_{j\geq i}\pa k_{ij}(x)\pa.
$$

The verification is analogous to the above Lemma: use
the category $\PNorm$ of pseudonormed spaces and linear maps of norm $\leq 1$. (The difference is that nonzero vectors can have norm $0$.)  This category is topological over the category of vector spaces and linear maps, and $\Norm$ is reflective in $\PNorm$ with reflections given by cokernel modulo the subspace of all vectors of norm $0$.
}
\end{rem}
 
\begin{lemma}\label{pres}
For an uncountable regular cardinal $\lambda$, a metric space is $\lambda$-presentable in $\Met$ iff it has a cardinality less than 
$\lambda$.
\end{lemma}

\begin{proof}
Every metric space is a $\lambda$-directed colimit of its subspaces of cardinality less than $\lambda$.
If $A$ is $\lambda$-presentable, the identity $\id_A:A\to A$ factorizes through one of these subspaces and thus $A$ has cardinality
less than $\lambda$. Conversely, let $A$ has cardinality less than $\lambda$ and let $k_i:K_i\to K$ be a $\lambda$-directed colimit of metric spaces $K_i$, $i\in I$, with connecting mappings $k_{ij}:K_i\to K_j$ for $i<j\in I$. Let $f:A\to K$ be a morphism. Since $U$
preserves $\lambda$-directed colimits, there is a mapping $f':A\to K_i$ such that $k_if'=f$. Given $a,b\in A$ then, following
\ref{directed}(1), $d(fa,fb)=d(k_{ij}fa,k_{ij}fb)$ for some $i<j\in I$. Since $A$ has less than $\lambda$ elements, there is $i<j\in I$
such that $k_{ij}f'$ is nonexpanding. Hence $A$ is $\lambda$-presentable.

\end{proof}

\begin{rem}\label{empty}
{
\em
(1) The only $\aleph_0$-presentable object in $\Met$ or $\CMet$ is the empty space. Indeed,
let $2_{\eps}$ be the two-element space with distance $\eps$ between the elements. The chain   
$$
2_1 \xrightarrow{\ \id\ } 2_{\frac{1}{2}} \xrightarrow{\ \id \ } 2_{\frac{1}{3}} \cdots
$$
has the one-point space $1$ as a colimit. Given a nonempty space $A$, the two distinct constant morphisms $f_1,f_2:A\to 2_1$ are not identified by any $\id:2_1\to 2_{\frac{1}{n}}$. Hence $A$ is not $\aleph_0$-presentable.

Analogously, the only $\aleph_0$-presentable object in $\Ban$ is the trivial space $0$: for every space $A$
use the sequence of spaces $A_n$ obtained from $A$ by dividing the norm by $1/n$, whose colimit is $0$.

(2) For $\lambda$ uncountable, $\lambda$-presentable objects in $\CMet$ are precisely complete metric
spaces of density character less than $\lambda$, i.e., those having a dense subset of cardinality less that $\lambda$ (see 
\cite{LR}). Similarly, $\lambda$-presentable objects in $\Ban$ are precisely the Banach spaces of density character less than $\lambda$.
}
\end{rem}

\section{$\Met$-enriched categories}

We consider categories enriched over the symmetric monoidal closed category $\Met$. Every such category $\ck$ has its underlying category $\ck_0$ and a metric is given on every hom-set $\ck_0(A,B)$ such that composition is nonexpanding. 

\begin{rem} \label{conical}
{
\em
We have to distinguish limits in $\ck_0$ from conical limits in $\ck$. The latter are those limits in $\ck_0$ which have a \textit{collectively isometric} limit cone. This means
a limit cone $p_i\colon P\to P_i$ \ ($i\in I$) such that given a parallel pair $u,v$
$$
\xymatrix@=3pc{
K\ar@<0.5ex>[r]^{u}
\ar@<-0.5ex>[r]_{v}& P\ar[d]^{p_i}\\
& P_i
}
$$
we have $d(u,v)= \sup\limits_{i\in I} d(p_iu, p_iv)$. As we see in Section 4 below these are precisely the weighted limits with the trivial weight. Observe that every object $K$ of $\ck$ yields
an obvious functor $\ck(K,-) \colon \ck_0\to \Met$ and that conical limits are precisely
those that each such hom-funktor preserves.

Analogously for conical colimits: this means that the colimit cocone $c_i:C_i\to C$ $(i\in I)$ fulfils for parallel pairs
$u,v:C\to K$ that $d(u,v)=\inf\limits_{i\in I}d(uc_i,vc_i)$.
}
\end{rem}

\begin{exam}
{
\em
The categories $\Met, \CMet$ and $\Ban$ have conical limits and colimits, see \ref{T:all}.
}
\end{exam}

\begin{defi}[\cite{RT}]\label{D:iso}
{
\em
 A morphism $f\colon K\to L$ is called  an \textit{isometry} if for every parallel pair 
$u,v \colon Q\to K$ we have
$$
d(u,v) = d(fu, fv)\,.
$$

Dually, $f$ is called a \textit{coisometry}  if $d(uf,vf)=d(u,v)$ for all $u,v:L\to Q$.
}
\end{defi}

\begin{exam}
{
\em
In $\Met$ isometries precisely represent inclusions of subspaces. 

Coisometries are precisely the morphisms with dense image. Indeed, every such morphism is clearly a coisometry. Conversely, let $f:A\to B$ be a coisometry and decompose it as $A\to \overline{f[A]}\to B$. Assume that $\overline{f[A]}\neq B$. Form a pushout
$$
\xymatrix@=2pc{
	A \ar [r]^{}\ar[d]_{} & \overline{f[A]} \ar[d]^{ v}\\
	\overline{f[A]}  \ar [r]_{{u}}& C
}
$$ 
This pushout replaces every element in $B\setminus\overline{f[A]}$ by a pair of points having distance $\infty$ (for this,
one needs that $\overline{f[A]}$ is closed in $B$). Then $uf=vf$ but $d(u,v)=\infty$, which is not possible.
}
\end{exam} 

\begin{rem}\label{isometry}
{\em
(1) Every isometry is a monomorphism.

(2) Let $\ck_0$ be an ordinary category and enrich it trivially over $\Met$ by putting $d(f,g)=\infty$ iff $f\neq g$. Then every monomorphism
in $\ck_0$ is an isometry in $\ck$. Thus, isometries need not be regular monomorphisms.

(3) A composition of two isometries is an isometry.

(4) Isometries are \textit{left cancellable}, i.e., if $gf$ is an isometry, then so is $f$. In fact, for
a parallel pair $u,v$ we have $d(fu,fv)\leq d(u,v)$. On the other hand, $d(u,v)=d(gfu,gfv)\leq d(fu,fv)$.

(5) $f:K\to L$ is an isometry iff $\ck(Q,f)$ is an isometry in $\Met$ for every $Q$ in $\ck$.

(6) Since $f$ is a coisometry in $\ck$ iff it is an isometry in $\ck^{\op}$, we have dual statements for coisometries.
}
\end{rem}

\begin{nota}\label{N:eps}
	{
		\em
		(1) For parallel morphisms $f$, $g\colon X\to Y$ we write $f\sim_\varepsilon g$ if their distance in $\ck(X,Y)$ is at most $\varepsilon$. (This corresponds well with \ref{N:Met})
		
		Analogously, we denote situations with $f\sim_{\varepsilon} g_2 \cdot g_1$  by triangles as follows
		$$
		\xymatrix@=1pc{
			\ar[rr]^f \ar[ddr]_{g_1}&&\\
			& \sim_\varepsilon&\\
			& \ar[uur]_{g_2}
		}
		$$
		
		(2) Note that $f\sim_{\eps} g$ in $\Ban$ iff $\parallel f - g\parallel\leq\eps$, which means that $\parallel fx - gx\parallel\leq\eps$
		for all $x\in X$, $\parallel x\parallel\leq 1$.
	}
\end{nota}

\begin{defi}[see \cite{RT}]\label{D:push}
	{\em 
	(1)	By an \textit{$\varepsilon$-pushout} of morphisms $f_i\colon A\to B_i$ ($i=1,2$) is meant a universal pair of morphisms
		$g_i\colon B_i \to C$ ($i=1,2$) with $g_1 f_1 \sim_\varepsilon g_2 f_2$.
 
		Universality is  in the usual (strict) sense: for every other such square $g_1' f_1 \sim_\varepsilon g_2' f_2$ (where 
		$g_i'\colon B_i\to C'$) there exists a unique $h\colon C\to C'$ with $g_1'=h\cdot g_1$ and $g_2'= h\cdot g_2$.
		
		(2) Analogously, an \textit{$\varepsilon$-coequalizer} of a parallel pair $f_1$, $f_2$ is a (strictly) universal morphism $c$ w.r.t.  
		$cf_1 \sim_\varepsilon f_2$.
That is every morphism $c'$ with $c'f_1\sim_\eps c'f_2$ uniquely factorizes through $c$.
		
		(3) The dual concepts are $\varepsilon$-pullback of a cospan and $\varepsilon$-equalizer of a parallel pair.	
}

\end{defi}

\begin{rem}\label{R:coeq}
{
\em
Every $\eps$-coequalizer is an epimorphism. This follows from the universal property.

It need not be a regular epimorphism. Denote by $2_\eps$ the two-point metric space with points of distance 
$\eps$. The $\eps/2$-coequalizer of the two points $p_1,p_2: 1 \to 2_\eps$ is the identity map from $2_\eps$
to $2_\eps/2$.
}
\end{rem}

\begin{exam}\label{weighted}
	{
		\em
$\Met$ has $\eps$-pushouts,   \cite[Lemma 2.3]{RT}. $\varepsilon$-pullbacks in $\Met$  are easy to construct:
		given morphisms $u,v$, form the following square
		$$
		\xymatrix@=2pc{
			D \ar [r]^{\bar u}\ar[d]_{\bar v} & C \ar[d]^{ v}\\
			B  \ar [r]_{{u}}& A
		}
		$$ 
		where $D$ is the subspace of $B\times C$ consisting of all pairs $(b,c)$ such that $d(ub,vc)\leq\varepsilon$ and $\bar u$, $\bar v$ are the projections.
	}
\end{exam}

\begin{defi} \label{D:equ}
{
\em
We say that $\ck$ \textit{has isometric  $\eps$-equalizers} if for every parallel pair an $\eps$-equalizer exists and is formed by an isometry.
	
Dually: $\ck$ has \textit{coisometric $\eps$-coequalizers}.
}
\end{defi}

\begin{exam}
{
\em
$\Met$  has isometric $\eps$-equalizers and coisometric $\eps$-coequalizers. Also $\CMet$ and $\Ban$ have this property.
This follows from \ref{T:all} and \ref{weighted1} below.
}
\end{exam}
	
\begin{defi}\label{D:extr}
		{
			\em
			A morphism $f\colon K\to L$ is called an \textit{isometry-extremal epimorphism} if every isometry $M\to L$ through which
			$f$ factorizes is an isomorphism.
			
			Dually, $f\colon K\to L$ is called an \textit{coisometry-extremal monomorphism} if every co\-iso\-met\-ry $K\to M$ through which
			$f$ factorizes is an isomorphism.
		}
	\end{defi}
	
\begin{rem}\label{iso-ext}
{
\em
(1)	If $\ck$ has conical equalizers, every isometry-extremal epimorphism $f:K\to L$ is an epimorphism. Indeed, assume that $uf=vf$ and let $e$ be an equalizer of $u$ and $v$. Since $e$ is an isometry, it is an isomorphism and thus $u=v$.
			
Dually, coisometry-extremal monomorphisms are monomorphisms if coequalizers are conical.

(2) If $\ck$ has finite conical products and isometric $\eps$-equalizers,
	then it has $\eps$-pullbacks which, moreover, are jointly isometric.
	This follows from the classical construction of a pullback of morphisms $f_i:
	A_i \to B$ for $i=1,2$  as an equalizer of the pair $f_1\pi_1,f_2\pi_2 : A_1
	\times A_2 \to B$. This
	applies to $\eps$-pulbacks, too,  whenever the product $A_1 \times A_2$
	with projections $\pi_i$ is conical.
	
(3) If $\ck$ has isometric $\eps$-equalizers for every $\eps>0$, then isometry-extremal epimorphisms are coisometries. Indeed,
 	let $f:K\to L$ be an isometry-extremal epimorphism and $u,v:L\to Q$. Let $e:E\to L$ be an $\eps$-equalizer of $u,v$ for
 	$\eps = d(uf,vf)$. Since $f$ factorizes through $e$ and $e$ is an isometry, $e$ is an isomorphism. Hence $d(u,v)=\eps$.
	
		}
	\end{rem}

		\begin{theo} \label{L:well} If $\ck_0$ is wellpowered and $\ck$ has conical limits, then
		it has the factorization system $(\ce,\cm)$ with $\cm 
			= \mbox{isometries}$ and $\ce=\mbox{isometry-extremal epimorphisms}$.
		\end{theo}
		
		\begin{proof}
			(1) We first verify that an intersection of  isometries $f_i \colon X_i \to Y$ \ ($i\in I$) is an isometry. Indeed, since our category is wellpowered and isometries are monic, an intersection $f$ exists:
			$$ 
			\xymatrix@=3pc{
				X\ar[r]^{\pi_i} \ar[d]_{f} &X_i\ar[dl]^{f_i}\\
				Y &}  \quad (i\in I)
			$$
			Given $u_1$, $u_2\colon Z\to X$, we have, since limits are conical,
			$$
			d(u_1, u_2) = \sup\limits_{i\in I} d(\pi_i u_1, \pi_i u_2)
			$$
			and since $f_i$ is an isometry,
			$$
			d(\pi_i u_1, \pi_i u_2) = d(f_i \pi_i u_1, f_i\pi_i u_2) = d( fu_1, fu_2)\,.
			$$
			This proves $d(u_1, u_2) = d(fu_1, fu_2)$.
			\vskip 2mm
			(2) For every morphism $g\colon X\to Y$ let $m:Z\to Y$ be the  intersection of all isometries through which $g$ factorizes. Then $m$ is an isometry and $g= mh$ for a unique $h\colon X\to Z$. Then $h$ is an isometry-extremal epimorphism. Indeed, suppose $h=m_0k$ for some  isometry $m_0 \colon Z_0 \to Z$. Then $mm_0$ is also an isometry, and since $g$ factorizes through it $(g= mm_0k)$, we see that $m$ is a subobject of $mm_0$, hence, $m_0$ is invertible as required.
			
			\vskip 2mm
			(3) The diagonal fill-in property holds. Indeed, given a commutative square as follows 
			$$
			\xymatrix@=3pc{
				A\ar[r]^{e} \ar[d]_{f} & B\ar[d]^{g} \\
				X \ar[r]_{m} & Y
			}
			$$
			with $m$ an isometry and $e$ an isometry-extremal epimorphism, form the pullback $P$ of $m$ along $g$:
			$$
			\xymatrix{
				A\ar[rr]^{e} \ar[dd]_{f} \ar@{.>} [dr]&& B\ar[dd]^{g} \\
				& P \ar[ur]^{\bar m} \ar[dl]_{\bar g}&\\
				X \ar[rr]_{m} && Y
			}
			$$
			Then $\bar m$ is an isometry: given a parallel pair $u_1$, $u_2\colon Z\to P$, then since the pullback is conical, we have
			$$
			d(u_1, u_2) = \sup\{d(\bar g u_1, \bar g u_2),d(\bar m u_1, \bar m u_2)\}\,.
			$$
			Since $m$ is an isometry, we conclude
			$$
			d(\bar g u_1, \bar g u_2) = d(m\bar g u_1, m\bar g u_2)
			= d(g\bar m u_1, g\bar m u_2)
			\leq d(\bar m u_1, \bar m u_2)\,.
			$$
			This proves 
			$$
			d(u_1, u_2) = d(\bar m u_1, \bar m u_2)
			$$
			as required.
			
			Since $e$ is an isometry-extremal epimorphism which factorizes through $\bar m$ (using the universal property of $P$), we conclude that $\bar m$ is invertible.
			The desired  diagonal is 
			$$
			\bar g \cdot \bar m^{-1} \colon B\to X\,.
			$$
		\end{proof}
		
\begin{rem}\label{fact}
{
\em
(1) Dually, we have the (coisometry, isometry-extremal monomorphism) factorization system, whenever $\ck_0$ is cowellpowered and $\ck$ has conical colimits. 
				
(2) We have shown that isometry-extremal epimorphisms coincide with isometry-strong epimorphisms, i.e., with those having the diagonal
				fill-in property w.r.t. isometries. Dually, isometry-extremal monomorphisms coincide with isometry-strong mo\-no\-mor\-phisms.
			}
		\end{rem}
		
\begin{exams}\label{fact1}
{
\em
(1) Isometry-extremal epimorphisms in $\Met$ are precisely the surjective morphisms. Thus we have (surjective, isometry) factorizations. Coisometry-extremal momomorphisms are precisely the closed isometries, i.e., isometries with a closed image. Since coisometries are precisely the dense morphisms, the second factorization system is
(dense, closed isometry).

(2) In $\CMet$ and $\Ban$ the two factorization systems coincide: we get the (dense, isometry) factorization system.
}
\end{exams}

\begin{lemma}\label{L:epush}
If $\ck$ has $\eps$-coequalizers for every $\eps$, then every limit is conical.
\end{lemma}

\begin{proof}
 Let $D$ be a diagram with a limit cone $p_i: P \to D_i$ $(i \in I)$. Given a parallel pair $u,u':Q \to P$, we are to prove 
	 $d(u,u') = \eps$, where $\eps$ is the supremum of  $d(p_iu,p_iu')$ for $i \in I$. Let $e:P \to R$ be an $\eps$-coequalizer of $u$ and $u'$, which is an epimorphism, see \ref{R:coeq}. Since for every $i$ we have $d(p_iu,p_iu') \leq \eps$, there are factorizations $p_i=f_ie$. These factorizations form a cone of
	 the diagram $D$ since $e$ is an epimorphism. Therefore, there exists $f: R \to P$ with $f_i=p_if$ for all $i \in I$. We conclude 
	 $fe=\id$, since for all $i$ we have $p_ife=p_i$. Therefore $e$ is an isomorphism, proving $d(u,u') = \eps$.
\end{proof}
	 
\begin{theo}\label{epush}
Let $\ck_0$ be complete, wellpowered and cowellpowered. Then $\ck$ has conical limits iff it has $\eps$-coequalizers for every $\eps$. If, moreover $\ck_0$ has finite coproducts and conical limits, then $\eps$-pushouts exist for all $\eps$.
\end{theo}

\begin{proof} 
(1) We first prove the last statement.

Given morphisms $f_i:A\to B_i$, consider an arbitrary $\varepsilon$-commutative square
	
	$$
	
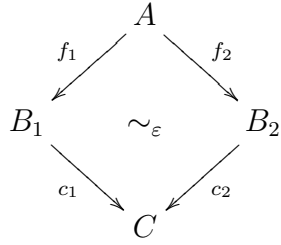
\xymatrix{ & A\ar[dl]_{f_1} \ar[dr]^{f_2} & \\
		B_1 \ar[dr]_{c_1} & {\sim_\varepsilon} & B_2\ar[dl]^{c_2}
		\\
		& C &
	}
	$$
	Factorize  $[c_1, c_2] \colon B_1+ B_2 \to C$ as an isometry-extremal epimorphism $\bar c \colon B_1 + B_2\to \bar C$ followed by an isometry $m_c\colon \bar C \to  C$ (see \ref{L:well}). Since $\bar c$ is an epimorphism (see \ref{iso-ext}(1)), these quotients have a set 
	$\bar c^t \colon B_1 + B_2 \to \bar C^t$ ($t\in T$) of representatives.
	
	Denote by $\bar P$ the product of all $\bar C^t$ with projections $\pi^t \colon \bar P\to \bar C^t$. The morphism
	$$
	\bar c =\langle \bar c^t\rangle_{t\in T} \colon B_1 + B_2 \to \bar P
	$$
	has a factorization as a isometry-extremal epimorphism $p\colon B_1 + B_2 \to P$ followed by an isometry $m_p\colon P\to \bar P$. Let $p_i \colon B_i \to P$ be the components of $p$, then we prove that the following square
	
	$$  
	
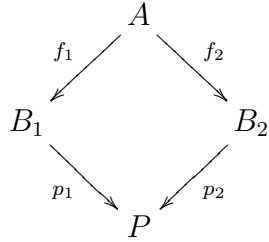
\xymatrix{ & A\ar[dl]_{f_1} \ar[dr]^{f_2} & \\
		B_1 \ar[dr]_{p_1} &  & B_2\ar[dl]^{p_2}
		\\
		& P &
	}
	$$
	is an $\varepsilon$-pushout.
	
	(a) We prove $p_1f_1 \sim_\varepsilon p_2 f_2$, i.e. $d(pv_1f_1, pv_2f_2)\leq \varepsilon$, where $v_i$ denotes the  injections 
	of $B_1 + B_2$. For every $t\in T$ we know that
	$$
	m_c \cdot \bar c^t \cdot v_1\cdot  f_1 \sim_\varepsilon m_c \cdot \bar c^t \cdot v_2\cdot f_2
	$$
	which, since $m_c$ is an isometry, implies
	$$
	\bar c^t \cdot v_1\cdot f_1 \sim_\varepsilon \bar c^t \cdot v_2\cdot f_2\,.
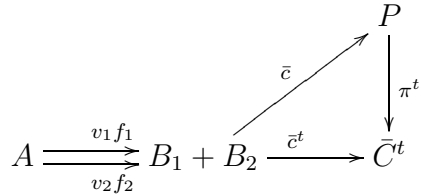
	$$
	 
	$$ 
	\xymatrix@=3pc{
		&& P \ar[d]^{\pi^t}\\
		A\ar@<0.5ex>[r]^{v_1 f_1}
		\ar@<-0.5ex>[r]_{v_2 f_2}& B_1+ B_2 \ar[ur]^{\bar c} \ar[r]^{\bar c^t} & \bar C^t
	}
	$$
The product $\bar P = \prod \bar C^t$ is conical, thus	$d(\bar c v_1 f_1, \bar c v_2 f_2)$ is the supremum of all $d( \bar c^t v_1 f_1, \bar c^t v_2 f_2)$, which proves 
	$$
	d(\bar c v_1 f_1, \bar c v_2 f_2) \leq \varepsilon\,.
	$$
	Since $\bar c = m_p p$ and $m_p$ is an isometry, this proves   the desired inequality
	$$
	d(p v_1 f_1, p v_2 f_2)\leq \varepsilon\,.
	$$
	
	(b) The factorization property needs only be verified for every pair $c_1^t$, $c_2^t$ ($t\in T$) since $\bar c^t$ represent all the above quotients $\bar c$ (and the isometries $m_c$ play no role). For every $t\in T$ the pair $c_1^t$, $c_2^t$ factorizes through $p_1$, $p_2$ since
	$$
	\bar c^t = \pi^t \cdot \bar c = \pi^t \cdot m_p\cdot p
	$$
	which precomposed with $v_i$ yields $c_i^t = (\pi^t \cdot m_p)\cdot p_i$.
	The factorization is clearly unique.
	
	(2) Conical limits imply $\eps$-coequalizers. This is completely analogous: given morphisms $f_1,f_2: A \to B$ consider the collection of all quotients
	$\bar c^t \colon B \to \bar C^t$ ($t\in T$) with $\bar c^t f_1\sim_\eps\bar c^t f_2$. Factorize 	$\bar c =\langle \bar c^t\rangle_{t\in T} $ as $m_Pp$ as above, then $p$ is the $\eps$-coequalizer of $f_1, f_2$.
	 
	 (3) $\eps$-coequalizers imply conical limits by the previous lemma.
	
\end{proof}

\begin{coro}\label{stable} 
If isometries are stable under pushouts, then they are stable under $\varepsilon$-pushout for every $\varepsilon>0$.
\end{coro}
\begin{proof}
The result follows from \ref{isometry}(4) and the fact that pushouts factorize through $\varepsilon$-pushouts.
\end{proof}

\begin{lemma}\label{coiso} 
Let $\ck_0$ be complete, wellpowered and cowellpowered, and let $\varepsilon$-equalizers and $\eps$-co\-equa\-li\-zers exist for all 
$\varepsilon >0$. Then the following conditions are equivalent:
\begin{enumerate}
\item Every isometry-extremal epimorphism is a coisometry, and
\item Every $\eps$-equalizer is an isometry.
\end{enumerate}
\end{lemma}
\begin{proof}
(1)$\Rightarrow$(2). Let $e:E\to K$ be an $\eps$-equalizer of $u,v:K\to L$. Factorize it as $e=gf$ where $f:E\to A$ is an isometry-extremal epimorphism and $g:A\to K$ is an isometry, see \ref{L:well}. Then $ug\sim_{\eps} vg$ and thus there exist $t:A\to E$ such that $g=et=gft$, hence
$tf=\id_E$. Since limits are conical by Lemma \ref{L:epush}, $f$ is an epimorphism (see \ref{iso-ext}(1)), thus, it is an isomorphism. Hence $e$ is an isometry. 

(2)$\Rightarrow$(1). Let $f:K\to L$ be isometry-extremal epimorphism and $u,v:L\to A$ such that $fu\sim_{\eps} fv$. Let $e:E\to K$ be an $\eps$-equalizer of $u$ and $v$. There exist $g:K\to E$ such that $eg=f$. Since $e$ is an isometry and $f$ is isometry-extremal, 
$e$ is an isomorphism. Hence $u\sim_{\eps} v$.
\end{proof}
 
\begin{defi}[\cite{K1}]\label{e-iso1}
{
\em
A morphism $f:A\to B$ is called an $\eps$-\textit{isometry} provided that there are isometries $g:B\to C$
and $h:A\to C$ such that $gf\sim_{\eps} h$.
}
\end{defi}

\begin{lemma}\label{e-iso2}
A morphism $f:A\to B$ is an $\eps$-isometry iff in the following $\eps$-pushout 
$$
\xymatrix@=3pc{
A\ar[r]^{f} \ar[d]_{\id_A} & B\ar[d]^{g} \\
A \ar[r]_{\overline{f}} & P
}
$$ 
$\overline{f}$ is an isometry.
\end{lemma}
\begin{proof}
If $\overline{f}$ is an isometry, then (since $g$ is also an isometry) $f$ is an $\eps$-isometry. Conversely, assume that $f$ is an $\eps$-isometry.
Then there are isometries $g:B\to C$ and $h:A\to C$ such that $gf\sim_{\eps} h$. Let $p:P\to C$ be the induced morphism.
Since $p\overline{f}=h$, $\overline{f}$ is an isometry (see \ref{isometry} (4)).
\end{proof}

\section{Weighted limits and colimits}
The appropriate concept of a (co)limit in a $\Met$-enriched category $\ck$ is that of a weighted (co)limit.
We are going to prove that the existence of ordinary (co)limits and (co)isometric $\eps$-(co)equalizers imply that weighted limits and colimits exist.

 Let us recall the concept of a limit of a diagram $D$ in $\ck$ weighted by a weight $W$. The diagram scheme is a small $\Met$-enriched category $\cd$, we denote by $[\cd,\Met]$
the enriched category of all enriched functors from it to $\Met$. That is, objects are all
(ordinary) functors $W \colon \cd \to \Met$ such that the induced map from $\cd(X,X')$ to
$\Met(WX,WX')$ is nonexpanding  for all pairs $X,X'$ in $\cd$. Morphisms are ordinary natural transformations. (They are automaticly enriched due to the fact that for the unit $1$ of the
monoidal category $\Met$ the hom-functor is faithful.)

 A \textit{ weight} is an enriched functor
$W: \cd \to \Met$.  Given a diagram in $\ck$, i.e. 
an enriched functor $D:\cd\to\ck$, its \textit{limit weighted by $W$} is an object $R$, usually written
	 $R=\lim_\cw D$, such that for all objects $K$ 
of $\ck$ we have isomorphisms
$$
\ck(K,R)\cong\ [\cd,\Met],(W,\ck(K,D-)).
$$
natural in $K$.  

\begin{exams} \label{weighted1}
{
\em

(1) Conical limits (see Remark \ref{conical}) are precisely the weighted limits with the trivial weight constant to the one-point space.
Here  $\cd$ is trivially enriched by putting all distances $\infty$ or $0$.

(2) For the one-morphism category $\cd$ a diagram is a choice of an object $L$ of $\ck$, and a weight
is a choice of a metric space $M$. The weighted limit is called the \textit{cotensor of $L$ and $M$}
and is denoted by $[M,L]$. It is characterized by the natural isomorphisms
$$
\ck(K,[M,L])\cong\Met(M,\ck(K,L)).
$$

(3) Let $\cd$ consist of a parallel pair with distance $\infty$. Then a diagram is precisely a parallel
pair of morphisms $u_1,u_2:K \to L$ in $\ck$. To express their $\eps$-equalizer, choose
the weight W given by the parallel pair consisting of two points $p_1, p_2:1 \to 2_{\eps}$, where $2_{\eps}$ is the space of two points with distance $\eps$. A weighted limit is then precisely an $\eps$-equalizer of $u_1,u_2$ which, moreover, is an isometry.

(4) Analogously for $\eps$-pullbacks. Let $\cd$ be a cospan and let $D$ correspond to a cospan $u_i:K_i \to L$ ($i=1,2$) in $\ck$. Here $W$ is given by the cospan  $p_1, p_2:1 \to 2_{\eps}$. Then a weighted limit is precisely an $\eps$-pullback $v_i:P \to K_i$ which, moreover, is collectively isometric, see Remark \ref{conical}.

Consequently, this weighted limit is nothing else than an $\eps$-equalizer of 
$$
\langle u_1\pi_1,u_2\pi_2\rangle: K_1 \times K_2 \to L
$$
formed by an isometry.
}
\end{exams}

\begin{rem} 
{
\em
The dual concept is that of a \textit{weighted colimit}. Here the weight is an enriched functor
	$W:\cd^{op} \to \Met$. A weighted colimit of a diagram $D$ is an object $R$, usually written
	 $R=\colim_\cw D$, with isomorphisms
$$
	\ck(R,K) \cong [\cd^{op},\Met]((W, \ck(D-,K))
$$
natural in $K$.	
	The dual construct to cotensor is \textit{tensor} of an object $L$ and a metric space $M$. It
is denoted by $M \otimes L$ and is an object with natural isomorphisms
$$
\ck(M \otimes L,K) \cong \Met(M, \ck(L,K)).
$$	
}
\end{rem}	

\begin{theo}[\cite{B}, 6.6.16]\label{weighted}
 A $\Met$-enriched category $\ck$ has weighted limits and colimits iff it has tensors and cotensors and $\ck_0$ is complete and cocomplete.
\end{theo}

\begin{assume}\label{A:enr}
Throughout the rest of the paper we assume that a $\Met$-enriched category $\ck$ is given with  the underlying category $\ck_0$ complete and cocomplete. And that $\ck$ has for all $\eps >0$ isometric $\eps$-equalizers and coisometric $\eps$-coequalizers.
\end{assume}

\begin{exams} \label{T:all}
\em

All of our running examples are categories with weighted limits and colimits (and thus satisfy Assumptions). Inded, they are all complete and cocomplete. Moreover:

(1) $\Met$ has tensors $M\otimes L$ and cotensors $[M,L]$ as described in Remark \ref{R:Met}(1).

(2) $\CMet$ has tensors $M\otimes L=M^\ast\otimes L$ where $M^\ast$ is the completion of $M$. This follows from the following isomorphisms natural in $K$ in $\Met$:
$$
\CMet(M^\ast\otimes L,K)\cong \CMet(M^\ast,\CMet(L,K))\cong \Met(M,\CMet(L,K)).
$$
Cotensors are $[M,L]=\Met(M^\ast,L)$, i.e., cotensors $[M^\ast,L]$ in $\Met$. Indeed,
$$
\CMet(K,[M^\ast,L])\cong \Met(M^\ast,\Met(K,L))\cong\Met(M,\Met(K,L)).
$$

(3) In order to describe tensors and cotensors in $\Ban$, denote by $U=\Ban(\Bbb C,-)$ the unit-ball functor to $\Met$.
Since $U$ preserves limits and $\aleph_1$-directed colimits, and $\Ban$ and $\Met$ are locally presentable, $U$ has a left adjoint 
$F:\Met\to\Ban$ (see \cite{AR} 1.66). Moreover, this adjunction is $\Met$-enriched.

The category $\Ban$ is symmetric monoidal closed where $\otimes$
is the projective tensor product, and internal hom is the space $\{K,L\}$ consisting of \textit{all} bounded linear mappings (not necessarily of norm at most 1) from $K$ to $L$ (see \cite{B} 6.1.9h). Observe that $U\{K,L\}=\Ban(K,L)$.

$\Ban$ has tensors $M\otimes L=FM\otimes L$ and cotensors $[M,L]=\{FM,L\}$. Indeed,
$$
\Ban(FM\otimes L,K)\cong \Ban(FM,\{L,K\})\cong \Met(M,U\{FM,L\})
$$
which is $\Met(M,\Ban(L,K))$. Similarly,
$$
\Ban(K,\{FM,L\})\cong \Ban(K\otimes FM,L)\cong\Ban(FM,\{K,L\}) 
$$
which is $\Met(M,\Ban(K,L))$.

\end{exams}

\begin{theo}\label{weighted2}
  All weighted limits and colimits exist in $\ck$.
\end{theo}

\begin{proof}
We only need to prove that cotensors exist. The dual result then yields tensors and we can apply the Theorem \ref{weighted}. 

Consider a metric space $ M$. We present cotensors for all objects $L$.

(1) If $M$ is the one-element space, then cotensors are trivial: $[M,L] = L$. 

(2) If $M$ has just two elements of distance $\eps$, then $\Met(M,\ck(K,L))$ is given by a parallel pair $u_1,u_2:K \to L$ in $\ck$ of distance $\eps$. Form the following $\eps$-pullback
$$
\xymatrix@=3pc{
	P_\eps\ar[r]^{v_\eps} \ar[d]_{u_\eps} & L\ar[d]^{\id_L} \\
	L \ar[r]_{\id_L} & L
}
$$
Then $u_\eps,v_\eps$ is a universal parallel pair of distance $\eps$. Indeed, 
these morphisms are collectively isometric by \ref{L:epush} and \ref{iso-ext}(2). 
Hence $P_\eps$ is the desired cotensor $[M,L]$ due to the natural isomorphism
$$
\ck(K,P)\cong\Met(M,\ck(K,L)).
$$

(3) Let $M$ be an arbitrary metric space. Form all subspaces $M_{x,y}$ on at most two elements $x,y \in M$.
Then $M$ is a canonical colimit of the diagram of all these subspaces and all inclusions $M_{x,x} 
\hookrightarrow M_{x,y}$, where the colimit maps are also the inclusions. By the above items we get a diagram of all 
cotensors $[M_{x,y},L]$ and all the derived morphisms  $[M_{x,y},L]  \to  [M_{x,x},L]$. The desired cotensor is then
a limit of this diagram:

$$
[M,L] = \lim_{(x,y) \in M \times M}  [M_{x,y},L]. 
 $$
Indeed, for every object $K$ the desired natural isomorphism is obtained as the following composite
$$
\ck(K,[M,L]) =\ck(K,\lim [M_{x,y},L])\cong\lim\ck(K,[M_{x,y},L])\cong\lim\Met(M_{x,y},\ck(K,L)),
$$
where we used the fact that limits are conical (\ref{epush}), thus preserved by
$\ck(K,-)$, and then we applied the universal property of cotensors. From this we get

$$
\ck(K,[M,L]) \cong\Met(\colim M_{x,y},\ck(K,L)) =  \Met(M,\ck(K,L)) ).
$$
\end{proof}

\section{Approximate Injectivity and Approximate Smallness}

Throughout this section we assume that a class $\ch$ of morphisms in a category
satisfying \ref{weighted2} (and thus having weighted limits and colimits) is given.
We now come to the central concepts of our paper: objects that are approximately injective or approximately small with respect to $\ch$.

In an ordinary category an object $X$ is called injective w.r.t.\ $\ch$ if for every member $h\colon A\to A'$ of $\ch$ all morphisms $f\colon A\to X$ factorize through $h$ (i.e. $f=f'h$ for some $f':A'\to X$). We denote by
$$
\Inj \ch
$$
the class of all such objects. Classes of objects of this form are called \textit{injectivity classes}. 

\begin{defi}[See \cite{RT}]\label{ain}
{\em
 (1) An object $X$ is \textit{approximately injective} w.r.t.\ $\ch$ if for every member $h\colon A\to A'$ of $\ch$ and every  morphism $f\colon A\to X$ there exist $\varepsilon$-factorizations $f'\colon A'\to X$ through $h$ for all $\varepsilon >0$:
$$
\xymatrix@=1pc{
A \ar[rr]^{h} \ar[ddr]_{f} && A' \ar[ddl]^{f'}\\
& \sim_\varepsilon&\\
& X &}
$$
 (2) The class of all these objects is denoted by
$$
\Inj_{\ap} \ch\,.
$$
An \textit{approximate injectivity class} is a class of objects of the form $\Inj_{\ap} \ch$. 
}
\end{defi}

\begin{rem}
{
\em
An object $X$ is approximately injective w.r.t. $h:A\to A'$ iff the induced morphism $\ck(h,X):\ck(A',X)\to\ck(A,X)$ is a coisometry (i.e. a dense mor\-phism of $\Met$).
This means that $X$ is $\ce$-injective in the sense of \cite{LaR} where $\ce$ is the class of coisometries in $\Met$.
}
\end{rem}

\begin{exam}
{
\em 
If $\ck$ is locally $\lambda$-presentable in the enriched sense (see Section 6), every approximate injectivity class is an injectivity class, as proved in \cite{RT}, but not conversely: see Example \ref{E:app}.
}
\end{exam}
 
\begin{rem}\label{R:inj-con}
{
\em 
Using the terminology from \cite{AHSo}, we say that a morphism $h$ is an \textit{injectivity consequence} of $\ch$ if $\Inj \ch \subseteq \Inj\{h\}$. That is, objects injective w.r.t.\ $\ch$ are also injective w.r.t.\ $h$. Here are two `approximate' versions:}
\end{rem}

\begin{defi}\label{D:cons}
{
\em
(1) A morphism $h$ is an \textit{approximate-injectivity consequence} of $\ch$
if $\Inj_{\ap}\ch \subseteq \Inj_{\ap}\{h\}$, i.e., objects approximately injective w.r.t.\ $\ch$ are also approximately injective w.r.t.\ $h$.  

(2) $h$ is called a \textit{strict approximate-injectivity consequence} of $\ch$ if $\Inj_{\ap} \ch \subseteq \Inj\{h\}$,
i.e., objects approximately injective w.r.t. $\ch$ are injective w.r.t $h$.  
}
\end{defi}

\begin{exam}\label{E:cons} 
{\em
 Given an $\varepsilon$-pushout

$$
  \xymatrix@=3pc{
 \ar[r]^h \ar[d]_{k} & \ar[d]^{k'} \\ 
     \ar [r]_{h'} &   
}  
$$
then $h'$ is a strict approximate-injectivity consequence of $h$.}
\end{exam} 

In fact, if $X$ is approximately injective w.r.t.\ $h$, then every morphism $u$ as in the  following diagram
$$
\xymatrix@C=3pc{
\ar[r]^h \ar[d]_{p} & \ar[d]^{p'} \ar@/^2pc/[ddd]^{v} \\ 
     \ar [r]^{h'}\ar[ddr]_{u} & \ar@{.>}[dd]^{w} \\
    &\\ & X
  }
$$
factorizes through $h'$. To see this, use the approximate injectivity of $X$ to choose a morphism $v$ with $up \sim_{\varepsilon} vh$. Then the universal property yields $w$ with $u=wh'$.

\begin{lemma}\label{L:cons} 
Let $h:A\to B$ and $h':A\to B'$ be morphisms. Given triangles
$$
  \xymatrix@=3pc{
  A \ar[r]^{h} \ar [d]_{h'} \ar[d]^<<<<{\ \ \sim_{1/n}} & B\\
  B' \ar[ur]_{f_n} &}
$$
$(n=1, 2, 3, \dots )$ then $h'$ is an approximate-injectivity consequence of $h$. 
\end{lemma}

\begin{proof}
Let $X$ be approximately injective w.r.t.\ $h$ and let $u\colon A\to B$ be given:
$$
\xymatrix@=3pc{
A\ar [r]^{h} \ar [d]_{h'}  \ar@/^-2pc/[dd]_{u}& B \ar [ddl]^v\\
B' \ar[ur]_<<<<<{f_n} \ar @{.>}[d]&\\
X &
}
$$
For every $\varepsilon >0$ choose $n$ with $\frac{2}{n} <\varepsilon$. We have a morphism $v$ with $u\sim_{1/n} v \cdot h$ which implies $u\sim_\varepsilon (v\cdot f_n) \cdot h'$, as desired. This follows from $u\sim_{1/n} v\cdot h$ and $v\cdot h \sim_{1/n} v\cdot f_n \cdot h'$, thus, $u\sim_{2/n} (v\cdot f_n)\cdot h'$.
\end{proof}

\begin{rem}\label{R:comp} 
{
\em
(1) Recall the concept of a \textit{transfinite composite} of morphisms:

(a) Given an ordinal $\alpha$, an $\alpha$-chain (of objects $K_i$, $i<\alpha$, and morphisms $k_{ji} \colon K_j \to K_i$ for 
$j\leq i$) is called \textit{smooth} if for every limit ordinal $i<\alpha$ we have
$$
K_i =\underset{j<i}{\colim}\ K_j
$$
with the colimit cocone $(k_{ji})_{j<i}$.

(b) Given a smooth $(\alpha +1)$-chain, the morphism $k_{0\alpha} \colon K_0\to K_\alpha$ is called the \textit{$\alpha$-composite} of the morphisms $(k_{i, i+1})_{i<\alpha}$.

Thus $\alpha=2$  yields the usual concept of a composite of two morphisms $K_0 \to K_1\to K_2$ up to isomorphism of the codomain $K_2$.

The case $\alpha=0$ means that every isomorphism is a $0$-composite (of the empty set of morphisms).

(c) Transfinite composites are $\alpha$-composites where $\alpha$ is an arbitrary ordinal.
If $\alpha$ has cofinality at least $\lambda$, then these chains is $\lambda$-directed.
We then speak about $\lambda$-directed transfinite composites.

(2) The closure of $\ch$ under pushout and transfinite composite is denoted by $\cell(\ch)$ (the \textit{cellular morphisms for 
$\ch$}). A well-known fact is that every object injective w.r.t.\  $\ch$  is also injective w.r.t. cellular morphisms for 
$\ch$ (see e.g. \cite{AHRT}).
}
\end{rem}

In ordinary categories an object is called $\lambda$-small if its hom-functor preserves
$\lambda$-directed transfinite composites of cellular morphisms. We are using the appropriate
enriched variant:

\begin{defi}\label{lambda-small}
{
\em	
 An object $A$ is called \textit{$\lambda$-small} w.r.t.\ $\ch$ if $\ck(A, -):\ck\to\Met$ preserves  $\lambda$-directed transfinite composites of cellular morphisms for $\ch$.

Explicitly: for every $\lambda$-directed smooth chain $k_{ij} \colon K_i\to K_j$ where $k_{i,i+1} \in \cell(\ch)$ 
($i<\mu$) with a colimit $k_i \colon K_i \to K_\mu$ ($i<\mu$) the following hold:
\begin{enumerate}
\item[(a)]given a morphism $f\colon A\to K_\mu$, there is $i<\mu$ such that $f$ factorizes through $k_i$: we have $f= k_if',$
\end{enumerate}

\begin{enumerate}
\item[(b)]   if $k_{i\mu}f'\sim_\varepsilon k_{i\mu}f''$ for $f',f'':A\to K_i$ and $i<\mu$, then 
$k_{ij}f'\sim_\varepsilon k_{ij}f''$ for some $i\leq j<\mu$.
\end{enumerate}
}
\end{defi}

Let us introduce the corresponding  approximate concepts:

\begin{defi}\label{D:cell} 
{
\em 
(1) The class of transfinite composites of $\varepsilon$-pushouts of morphisms 
from $\ch$ (for all $\varepsilon>0$) is denoted by $\cell_{\ap}(\ch)$. Its members are called \textit{approximately cellular} morphisms for $\ch$.

(2) Let $\lambda$ be a regular cardinal. An object $A$ is called \textit{approximately  $\lambda$-small} w.r.t. $\ch$
if for every $\lambda$-directed transfinite composite $(k_{ij}:K_i\to K_j)_{i\leq j\leq\mu}$ of morphisms approximately cellular
for $\ch$ and every $\varepsilon>0$ we have that 
\begin{enumerate}
\item[(a)]  every morphism $f:A\to K_\mu$ has an $\varepsilon$-factorization $f':A\to K_i$ through $k_{i\mu}$ for some $i<\mu$,  i.e.\ $f\sim_\varepsilon k_{i\mu}f'$,
\item[(b)]  given $i<\mu$ with $k_{i\mu}f'\sim_\varepsilon k_{i\mu}f''$ for $f',f'':A\to K_i$, then $k_{ij}f'\sim_\varepsilon k_{ij}f''$ for some $i\leq j<\mu$.
\end{enumerate}
}
\end{defi}

\begin{rem}\label{small:iso}
{
\em

(1) Every $\lambda$-small object w.r.t. $\ch$ is approximately $\lambda$-small w.r.t. $\ch$. 

(2) If $\cell_{\ap}(\ch)$ consists of isometries, then condition (b) can be omitted in \ref{D:cell}.

(3) If $\ck$ is enriched over $\CMet$, then $A$ is approximately $\lambda$-small w.r.t. $\ch$ iff $\ck(A,-):\ck\to\CMet$ preserves
$\lambda$-directed transfinite composites of approximately cellular morphisms for $\ch$.
}
\end{rem}

From \ref{R:comp}(2) and \ref{E:cons} we conclude the following fact.

\begin{lemma}\label{L:cell}
All approximately cellular morphisms are strict injectivity consequences of the given class $\ch$.
\end{lemma}

\begin{lemma}\label{L:copr}
A coproduct of less than $\lambda$ approximately $\lambda$-small objects w.r.t. $\ch$ is approximately $\lambda$-small w.r.t. $\ch$, i.e., if $h\in\cell_{\ap}(\ch)$ then $\Inj\ch\subseteq\Inj\{h\}$.
\end{lemma}

\begin{proof}
Let $u_t:A_t\to\coprod_{t\in I} A_t$ with $A_t$ approximately $\lambda$-small  where $|I|<\lambda$, and let a morphism 
$f:\coprod A_t\to K_\mu$  be given. Since $|I|<\lambda$ and $(k_{ij}:K_i\to K_j)_{i\leq j\leq\mu}$ is $\lambda$-directed, there exist $i<\mu$ such that for every $t\in I$ we have $f'_t:A_t\to K_i$ 
with $k_{i\mu}f'_t\sim_\varepsilon fu_t$. Let $f':\coprod A_t\to K_i$ be the induced morphism, i.e., $f'u_t=f'_t$ for every $t\in I$. Since coproducts are conical, we have $k_{i\mu}f'\sim_\varepsilon f$. Given morphisms $f',f''\colon \coprod A_t \to K_i$ with  
$k_{i\mu}f''\sim_\varepsilon k_{i\mu}f'$, we conclude $k_{i\mu}f'u_t\sim_\varepsilon k_{i\mu}f''u_t$ for each $t\in I$. There is 
$j\geq i$ such that $k_{ij}f'u_t\sim_\varepsilon k_{ij}f''u_t$ for each $t\in I$. Hence $k_{ij}f'\sim_\varepsilon k_{ij}f''$.
\end{proof}

\begin{theo}\label{T:cell} 
For every uncountable regular cardinal $\lambda$ and every $\varepsilon>0$ all approximately $\lambda$-small objects w.r.t. $\ch$ are stable under $\varepsilon$-pushouts.
\end{theo}

\begin{proof}
Let us fix a number $\eps_0>0$ and let
\newcommand{\pullbackcorner}[1][dr]{\save*!/#1+1.2pc/#1:(1.5,-1.5)@^{|-}\restore}
$$
  \xymatrix@C=3pc@R=.51pc{
 A\ar[r]^u \ar[ddd]_{v} & B \ar[ddd]^{\bar v} \\ 
 &&&&\\
 &&&&\\
  C   \ar [r]_{\bar u}\ar[r]^<<<<<<<<<<<<<{{\!\varepsilon_0}} 
 & \ D \pullbackcorner
 }
 $$
be an $\varepsilon_0$-pushout with $A$, $B$ and $C$ approximately $\lambda$-small w.r.t. $\ch$. For a morphism $f\colon D\to K_\mu$ where
$(k_{ij}:K_i\to K_j)_{i\leq j\leq\mu}$ is given as in \ref{D:cell}(2) and given $\varepsilon>0$ we verify Conditions (a) and (b) 
for $f$.

Condition (a).  Since $C$ satisfies (a) we can find for $f\cdot \bar u \colon C \to K_\mu$ a $\frac 1n$-factorization $g_1$ through some $k_{i\mu}$ ($i<\mu$). Analogously for $f\cdot \bar v$ a $\frac1n$-factorization $g_2$ (through the same $k_{i\mu}$):
\renewcommand{\pullbackcorner}[1][dr]{\save*!/#1+1.2pc/#1:(1,-1)@^{|-}\restore}
\begin{equation}\label{e1}
  \xymatrix@C=3pc@R=.51pc{
 A\ar[r]^u \ar[ddd]_{v} & B \ar[ddd]^{\bar v} \ar[dddddrr]^{g_2}_{\sim_{1/n}}\\ 
 &&&&\\
 &&&&\\
  C   \ar [r]_{\bar u}^<<<<<<<<<<<<<<<<{{\!\!\!\!\!\varepsilon_0}} \ar[ddrrr]^{\!\!\!\!\!\!\sim_{1/n}}_{g_1}
 &\ D \pullbackcorner 
\ar[dr]^<<<{\,\,f} &&\\
&&K_\mu&\\
&&& K_i \ar[ul]_{\!\!\!k_{i\mu}}
 }
 \end{equation}

 We conclude for $\bar \varepsilon =\varepsilon_0 +\frac 2n$ that
 $$
 k_{i\mu} \cdot g_1\cdot  v \sim_{\bar\varepsilon} k_{i\mu}\cdot g_2\cdot u
 $$
 due to
 $$
 k_{i\mu}\cdot g_1\cdot  v \sim_{1/n} k_{i\mu}\cdot \bar u \cdot v \sim_{\varepsilon_0} k_{i\mu} \cdot \bar v \cdot u \sim_{1/n} 
 k_{i\mu}\cdot g_2\cdot u\,.
 $$
 Since $A$ is approximately $\lambda$-small, Condition (b) in \ref{D:cell}(2) implies that there exists $i\leq j<\mu$ such that $k_{ij}g_1u\sim_{\bar{\eps}}k_{ij}g_2v$:
\begin{equation}\label{e2}
\xymatrix@C=0.5pc{
A \ar[r]^u \ar[d]_v & B\ar[r]^{g_1} & K_i\ar[dd]^{k_{ij}}\\
C\ar[d]_{g_2}&\hskip 1cm \sim_{\varepsilon_0+\frac 2n}  & \\
K_i \ar[rr]_{k_{ij}} &&K_j
}
\end{equation}
This ordinal $j$ can be chosen independent  of $n$ ($\in \mathbb N$) since $\lambda$ is uncountable and regular. Let us form an $\bar \varepsilon$-pushout $D_n$ of $u$ and $v$:

\renewcommand{\pullbackcorner}[1][dr]{\save*!/#1+1.2pc/#1:(1.5,-1.5)@^{|-}\restore}
\begin{equation}\label{e3}
  \xymatrix@C=3pc@R=.51pc{
 A\ar[r]^u \ar[ddd]_{v} &\ B \ar[ddd]^{\bar v_n} \ar@{-->}[ddddr]^{k_{ij}\cdot g_1}&\\ 
 &&&&\\
 &&&&\\
  C   \ar [r]^{\bar u_n}^<<<<<<<<<<<<<{{{\bar \varepsilon}}} \ar@{-->}[drr]_{k_{ij}\cdot {g_2}}
 &\ D_n \pullbackcorner
\ar[dr]^<<<{\,\,h_n} &\\
&&K_j
}
\end{equation}
We obtain the unique factorization $h_n\colon D_n \to K_j$ as  indicated above.

The objects $D_n$, $n\geq 1$ form an $\omega$-chain, where $d_{nm} \colon D_n \to D_m$ is the obvious morphism we get from $\varepsilon_0 +\frac 1n \geq \varepsilon_0 +\frac 1m$ (for all $n\geq m$). The object $D$ is a colimit of this chain with colimit morphisms $d_n\colon D_n \to D$ uniquely determined by
\begin{equation}\label{e4}
\bar u = d_n \cdot \bar u_n\quad \mbox{and} \quad \bar v =d_n\cdot \bar  v_n\ \mbox{\  for all \ $n<\omega$}\,.
\end{equation}
And the morphisms $h_n$ form a cocone: if $n\leq m$ then $h_n = h_m \cdot d_{nm}$ because  
$$
h_n \cdot \bar{u}_n= k_{ij}\cdot g_2 = h_m \cdot \bar{u}_m = (h_m\cdot d_{nm})\cdot \bar{u}_n
$$
and analogously for $v_n$. Thus we get a unique
\begin{equation}\label{e5}
 f'\colon D\to K_j \quad \mbox{with}\quad f' \cdot d_n = h_n \quad (n\geq 1)\,.
 \end{equation}
 This is the desired $\varepsilon$-factorization of $f$: choose $n$ with $\frac 1n <\varepsilon$, then 
 $f\sim_\varepsilon k_{j\mu} \cdot f'$ because $f\sim_{1/n} k_{j\mu} \cdot f'$. Indeed,
to verify this last statement, we only need proving that  
$f\cdot \bar u\sim_{1/n} k_{j\mu} \cdot f'\cdot \bar u$ (and analogously for $\bar v$):
 \begin{align*} 
 f\cdot \bar u &\sim_{1/n} k_{i\mu} \cdot g_1 && \mbox{see \eqref{e1}}&&\\
  & = k_{j\mu} \cdot k_{ij} \cdot g_1 \\
  &= k_{j\mu} \cdot h_n \cdot \bar v_n && \mbox{see\eqref{e3}}&&\\
  & = k_{j\mu}\cdot f' \cdot d_n \cdot\bar  v_n && \mbox{see \eqref{e5}}&&\\
  &= k_{j\mu} \cdot f' \cdot \bar u && \mbox{see \eqref{e4}}.&&
  \end{align*}
  
Condition (b). Assume that for $i<\mu$ we have $k_{i\mu}\cdot f'\sim_\varepsilon k_{i\mu}\cdot f''$ for $f',f'':A\to K_i$. Since $C$ is approximately
$\lambda$-small and 
  $$
  f \cdot \bar u \sim_{1/n} k_{j\mu} \cdot f'\cdot \bar u\,, \quad f\cdot \bar u \sim_{1/n} c_j\cdot f'' \cdot \bar u
  $$
imply $k_{j\mu} \cdot f'\cdot \bar u\sim_{\frac 2n} k_{j\mu}\cdot f''\cdot \bar u$ (and analogously for $\bar v$), there exists $\bar j<\lambda$ such that $k_{j\bar j} \cdot f'\cdot \bar u \sim_{\frac 2n} k_{j\bar j} \cdot f''\cdot \bar u$ (and analogously for $\bar v$). Choose $n$ with $\varepsilon\geq \frac 2n$ to get  $k_{j\bar j}\cdot f' \cdot \bar u \sim_\varepsilon  k_{j\bar j} \cdot f''\cdot \bar u$ as well as $k_{j\bar j}\cdot f' \cdot \bar v \sim_\varepsilon k_{j\bar j} \cdot f'' \cdot \bar v$. By the dual of Example \ref{weighted1}(4), $\eps_0$-pushouts are (under our standing assumptions) weighted colimits, therefore the morphisms $\bar{u},\bar{v}$ are collectively coisometric. Thus we have proved that $k_{j\bar{j}}f'\sim_\eps k_{j\bar{j}}f''$, as desired.
\end{proof}

\begin{rem}
	{
		\em 
		For $\lambda=\aleph_0$ \ref{T:cell} does not hold, see \ref{counter}.  
	}
\end{rem} 
  
\begin{defi}\label{apgen}
{
\em  
For a regular cardinal $\lambda$, an object $A$ in $\ck$ will be called \textit{approximately $\lambda$-generated} if it is approximately $\lambda$-small w.r.t. the class of all isometries.
}
\end{defi} 

\begin{rem}\label{apgen1}
{
\em
The above concept was introduced in \cite{RT} 6.4 where, however, instead of $\lambda$-directed composites, general $\lambda$-directed colimits were used. For $\lambda = \aleph_0$ this makes no difference provided that isometries are closed under directed colimits in the category $\ck^\to$ of morphisms of $\ck$ (see \cite{AR} 1.7). This is true in $\Met$, $\CMet$ and $\Ban$.
}
\end{rem}

\begin{exams}
{\em 
(1) Finite metric spaces are  approximately $\aleph_0$-generated in $\Met$ (use \ref{directed}(2)).  

(2) Finite discrete spaces are approximately $\aleph_0$-generated in $\CMet$. Just use the fact that a directed colimit in $\CMet$
is a completion of that in $\Met$. Conversely:

}
\end{exams}

\begin{propo}
Every finite space which is approximately $\aleph_0$-generated in $\CMet$ is discrete.
\end{propo}

\begin{proof}
If a finite space $A$ is not discrete, we prove that it is not approximately $\aleph_0$-generated. Let $d$ be the minimum distance of distinct elements of $A$, then $0<d<\infty$. Choose $x$, $y\in A$ of distance $d$.
Consider the following subspace $K$ of the real line
$$
K= \{0,d\} \cup \Big\{d+\frac1n; n=1,2,3,\dots \Big\}\,.
$$
This is a colimit of the $\omega$-chain of its subspaces $K_r =\{0\} \cup \big\{ d+\frac 1n; n=1,2,\dots , r\big\}$ for $r=1,2,3\dots$ in $\CMet$. The function $f\colon A\to K$ mapping $x$ to $0$ and all other points to $d$ is clearly nonexpanding. Let $\eps<d$.
Then for no $r$ is there a nonexpanding $\eps$-factorization $g$ of $f$ through the colimit map $K_r \hookrightarrow K$ (since $d(x,y) =d$, we must have $g(x)=0$ and there  is no point  of $K_r\backslash \{0\}$ of distance at most  $d$ from $0$, thus, no possibility for $g(y)$). Therefore, $A$ is not approximately $\aleph_0$-generated.
\end{proof}

\begin{rem}\label{counter}
	{
		\em 
		For $\lambda=\aleph_0$ \ref{T:cell} does not hold. Consider the $\varepsilon$-pushout below in $\CMet$:
		\newcommand{\pullbackcorner}[1][dr]{\save*!/#1+1.2pc/#1:(1.5,-1.5)@^{|-}\restore}
		$$
		\xymatrix@C=3pc@R=.51pc{
			1\ar[r]^{\id} \ar[ddd]_{\id} & 1\ar[ddd] \\ 
			&&&&\\
			&&&&\\
			1  \ar[r]^<<<<<<<<<<<<<<<<{{\!\!\!\!\!\varepsilon}} 
			& \ D \pullbackcorner 
		}
		$$
		Then $D$ is a space of two points of distance $\varepsilon$. Although $1$ is approximately $\aleph_0$-generated, $D$ is not.  
	}
\end{rem} 
 
\begin{rem}\label{almost}
{\em 
Every finite space $A$ is `almost' approximately $\aleph_0$-generated in $\CMet$ in the following sense:
let $(k_{ij}:K_i\to K_j)_{i\leq j\leq\mu}$ a transfinite composite of isometries in $\CMet$
and let $f:A\to K_\mu$ be an isometry. Then for every $\varepsilon>0$ there is an $\varepsilon$-isometry $f':A\to K_i$ in the sense of \ref{e-iso} (1)
(not necessarily nonexpanding) with  $i< \mu$ such that $k_if'\sim_\varepsilon f$. In fact, there is $i<\mu$ such that
$$
\varepsilon\geq |d(k_if'x,k_if'y)-d(fx,fy)|=|d(f'x,f'y)- d(x,y)|.
$$
}
\end{rem}
 
\begin{lemma}
Let $\lambda$ be an uncountable regular cardinal and  $\ck$  be  $\CMet$-enriched.  Then every  approximately $\lambda$-small object 
$A$  is $\lambda$-small.
\end{lemma}
\begin{proof}
We have to check condition (a) of \ref{lambda-small}. Let 
$(k_{ij}:K_i\to K_j)_{i\leq j\leq\mu}$ be a $\lambda$-directed transfinite composite of morphisms from $\ch$, and let  $f:A\to K_\mu$ be a morphism. For every $n>0$, there is $f'_n:A\to K_{i_n}$ such that $k_{i_n\mu}f'_n\sim_{\frac{1}{n}} f$. Hence $k_{i_n\mu}f'_n\sim_{\frac{1}{n}+\frac{1}{m}} k_{i_m\mu}f'_m$. Following \ref{directed}(1), there is $i>i_n$ for every $n$ such that $k_{i_ni}f'_n\sim_{\frac{1}{n}+\frac{1}{m}} k_{i_mi}f'_m$ (for all $n$, $m$). Thus $k_{i_ni}f'_n$ form  a Cauchy sequence in $\ck(A,K_i)$. Let $f':A\to K_i$ be its limit. Then $k_{i\mu}f'$ is the limit of $k_{i\mu}k_{i_ni}f'_n$ and thus $k_{i\mu}f'=f$.
\end{proof}

\begin{defi} 
{
\em
(1) A morphism $m\colon K\to L$ is called \textit{approximately split} if for every $\varepsilon >0$ there is a morphism $e\colon L\to K$ with $e\cdot m\sim_\varepsilon \id_K$.

(2) A full subcategory $\cl$ of  $\ck$ is \textit{closed under approximately split morphisms} if for every approximately split morphism $K\to L$ with $L$ in $\cl$ one also has $K$ in $\cl$.
}
\end{defi}

\begin{rem} \label{spl}
{
\em
Every approximate injectivity class is closed under approximately split morphisms (see \cite{RT} 5.5 and 3.4(3)).
}
\end{rem}

\begin{exam}\label{E:app}
{
	\em 
	The class $\CMet$ of complete metric  spaces is an injectivity class in $\Met$ which fails to be an approximate injectivity class.
	
	(1) To verify that $\CMet$ is not an approximate injectivity class in $\Met$, observe that in the real line the inclusion $m:(0,1]\to[0,1]$ is approximately split. Indeed, $g_n\colon  [0,1] \to (0,1]$ defined by $g_n(x) = x+\frac{1-x}{n}$ is  nonexpanding and fulfils $g_n\cdot m \sim_{1/n} \id$. Indeed, for every $x$ it fulfils $d\big(x,g_n(x)\big)\leq \frac 1n$ because $d^2\big(x, g_n(x)\big) =\frac{(x-1)^2}{n^2} < \frac{1}{n^2}$.
Thus, $\CMet$ is not closed under approximately split morphisms .

(2) $\CMet$ is an injectivity class: For every countable metric space $A$ denote by $i_A \colon A\to A^\ast$ its Cauchy completion. Then $\CMet$ is  the injectivity class of the (essentially small) set of all $i_A$'s. Indeed, a Cauchy sequence in a space $X$ has a limit  in $X$ iff for the embedding $f\colon A\hookrightarrow X$ of the  corresponding (countable) subspace a non-expanding extension to $A^\ast$ exists.

}

\end{exam}

\begin{exam}\label{finite}
{
\em
(1) Finite spaces are closed in $\Met$ under approximately split morphisms.
In fact, assume that an infinite metric space $B$ has an approximately split morphism $u:B\to A$ to a finite space $A$, we derive a contradiction. There are morphisms 
$s_n:A\to B$ such that $s_nu\sim_{\frac{1}{n}}\id_B$. Choose pairwise distinct elements $b_i\in B$, $i<\omega$ with $u(b_i)=u(b_j)$ for every $i,j<\omega$. 
Then given $i\ne j$ for every $n$ we conclude $b_i\sim_{2/n} b_j$:
$$
b_i\sim_{\frac{1}{n}} s_nu(b_i) = s_nu(b_j)\sim_{\frac{1}{n}} b_j.
$$
Hence $ b_i=b_j$, which is the desired contradiction.

(2) Separable metric spaces are also closed in $\Met$ under approximately split morphisms. Let $u:B\to A$ be approximately split and $B$ separable. Let $X$
be a countable dense subset of $A$. For $s_n$ as in (1) we have a countable dense subset $\bigcup\limits_{n<\omega} s_n(X)$ of $A$.
}
\end{exam}

\begin{propo}\label{split}
Let  $\lambda$ be  an uncountable regular cardinal. Then approximately $\lambda$-small objects w.r.t. $\ch$ are closed under approximately split morphisms.
\end{propo}

\begin{proof}
Let $u:A\to B$ be an approximately split morphism with $B$ approximately $\lambda$-small w.r.t. $\ch$.  Consider  
$(k_{ij}:K_i\to K_j)_{i\leq j\leq\mu}$ from \ref{D:cell}(2). Choose a morphism $f:A\to K_\mu$. For every $\varepsilon >0$ and arbitrary morphisms $t:B\to A$ and $f':B\to K_i$ there is a  morphism $u$ such that $tu\sim_{\frac{\varepsilon}{2}}\id_A$ and 
$k_{i\mu}f'\sim_{\frac{\varepsilon}{2}} ft$. Hence $k_{i\mu}f'u\sim_{\frac{\varepsilon}{2}} ftu\sim_{\frac{\varepsilon}{2}} f$ and thus 
$k_{i\mu}f'u\sim_\varepsilon f$. This verifies (a) in Definition~\ref{D:cell}. To verify  (b), assume that $k_{i\mu}f'\sim_\varepsilon k_{i\mu}f''$ for $f',f'':A\to K_i$. For $n>0$, there are morphisms $t_n:B\to A$ such that $t_nu\sim_{\frac{1}{n}}\id_A$
and we choose $j_n\geq i$ such that $k_{ij_n}f't_n\sim_\varepsilon k_{ij_n}f''t_n$. Hence
$$
k_{ij_n}f'\sim_{\frac{1}{n}} k_{ij_n}f't_nu\sim_\varepsilon k_{ij_n}f''t_nu\sim_{\frac{1}{n}} k_{ij_n}f''
$$
and thus $k_{ij_n}f'\sim_{\varepsilon +\frac{2}{n}} k_{ij_n}f''$. 
Since $\lambda$ is uncountable, we can choose
 $j>j_n$ for all $n$, and we get $k_{ij}f'\sim_\varepsilon k_{ij}f''$.
\end{proof}

\begin{theo}\label{T:ref}
Assume  that the domains of morphisms from $\ch$ are approximately small w.r.t. $\ch$. Then $\Inj_{\ap}\ch$
is weakly reflective in $\ck$  with approximately cellular weak reflections.
\end{theo}
  
 \begin{proof}
 For every object $K$ we construct a weak reflection in $\Inj_{\ap} \ch$ in two transfinite steps: we first define a morphism $t_K \colon K\to K^\ast$ as a composite of a certain transfinite chain formed by $\varepsilon$-pushouts of morphisms of $\ch$. Then we iterate this step from $K$ to $K^\ast$ transfinitely in order to obtain the desired weak reflection $\widehat K$.
 
Construction (1):
 Consider the set $\cx_K$  of all triples $(h,u,n)$ where $u$ and $h$ form a span
$$
\xymatrix@=3pc{ 
A \ar [r]^h \ar[d]_u & B\\K&
}
$$
with $h\in \mathcal H$ and $n>0$ is a natural number. Put $\mu_K = \card \cx _K$. We will index members of $\cx_\ck$ by ordinals 
$i < \mu_\ck$. That is, $\cx_K =\{(h_i, u_i, n_i); i<n_K\}$. We define a chain $k_{ij} : K_i \to K_j$,
$i\leq j \leq \mu_K$
and morphisms $b_i\colon B_i\to K_{i+1}$
 by the following transfinite recursion:

First step: $K_0 =K$.

Isolated step: $K_{i+1}$, $r_i$ and $k_{i, i+1}$ are given by an $\frac{1}{n_i}$-pushout as follows
\newcommand{\pullbackcorner}[1][dr]{\save*!/#1+1.5pc/#1:(1,-1)@^{|-}\restore}
$$
\xymatrix@C=4pc@R=2pc{
A_i \ar [r]^{h_{i}} 
\ar[d]_{u_i} & B_i\ar [dd]^{r_i}\\
K\ar[d]_{k_{0i}}&\\
K_i\ar[r]_{k_{i, i+1}} \ar[r]^<<<<<<<<<<<<<<<<<{^{_{^{\quad {1}\!/\!{n_i}}}}}&\ 
 K_{i+1} \pullbackcorner 
 } 
$$
We put $k_{j,i+1} = k_{i, i+1} \cdot  k_{j,i}$.\ for all $j<i$.

Limit step: $K_i$ is the colimit of the chain $(K_j)_{j<i}$ and $k_{ji} : K_j \to K_i$ are given as the colimit cocone for all $j<i$.

The object $K_{\mu_K}$ will be denoted by $K^\ast$ and the morphism $k_{0\mu_k} : K\to K^\ast$ by $t_K$. For every $i$ we obtain the following $\frac{1}{n_i}$-commutative square
$$
\xymatrix@C=4pc@R=2pc{
A_i \ar [r]^{h_{i}} \ar[dd]_{u_i} & B_i\ar [d]^{r_i}\\
& K_{i+1}\ar[d]^{k_{i+1,\mu_K}}\\
K \ar[r]_{t_K} & K^\ast
 }
$$

Construction (2): We are ready to construct a weak reflection $\widehat K$ of $K$. 
By assumption on $\ch$ there is a regular cardinal $\lambda$ such  that the domains of all morphisms in $\ch$ are $\lambda$-small. We define a $\lambda$-chain by iterating Construction (1) $\lambda$-times. That is, we define 
$m_{ij} : M_i \to M_j$, $i\leq j \leq \lambda$ by the following  transfinite recursion:

First step: $M_0 = K$.

Isolated step: $ M_{i+1}= M_i^\ast$ and $m_{i, i+1}=t_{M_i}$.

Limit step:  $M_i$ is the colimit of the chain $(M_j)_{j<i}$ with the colimit cocone $(m_{ji})_{j<i}$. In particular,  $M_\lambda$ is a colimit of $(M_j)_{j<\lambda}$.  We put
$$ 
\widehat K = M_\lambda \quad \mbox{and}\quad r_K = m_{0,\lambda}\colon K\to \widehat K\,.
$$

We will show that this is a desired weak reflection of $K$.
 
(2a) $\widehat K$ is approximately injective. Indeed, given $h:A\to B$ in $\ch$, a morphism $u:A\to \widehat K$ and $n>0$, we provide an $\frac{1}{n}$-factorization of $u$ through $h$ as follows. 
Since the object $A$ is $\lambda$-approximately small and $\widehat K$ is a directed colimit of $M_i$, $i < \lambda$, there is a $\frac{1}{2n}$-factorization
$$
\xymatrix@=3pc{
&M_i \ar[d]^{m_{i\lambda}}\ar[d]_>>{\sim_{1/2n}}\\
A \ar[ur]^{u'} \ar[r]_{u}& \widehat K
}
$$
of $u$ through $m_{i\lambda}$ for some $i<\lambda$ (i.e. the triangle is $\frac{1}{2n}$-commuting). Since the triple $(h,u',2n)$ 
lies in the set $\cx_{M_i}$, we obtain a $\frac{1}{2n}$-commutative square as follows
$$
\xymatrix@=4pc{\ar@{} [dr]|{\sim_{1/2n}}
A \ar [r]^{h} \ar[d]_{u'} &  B\ar[d]^{v}\\
M_i\ar[r]_{m_{i, i+1}} &M_{i+1}
}
$$
We have
$$
u\sim_{\frac{1}{2n}} m_{i\lambda}u^\prime = m_{i+1,\lambda} m_{i, i+1}u^\prime \sim_{\frac{1}{2n}} m_{i+1,\lambda}vg\,.
$$
Hence $u\sim_{\frac{1}{n}}m_{i+1,\lambda}vg $, which proves that  $\widehat K$ is approximately injective w.r.t. $\ch$.

(2b) Since $m_{0\lambda}:K\to  \widehat K$ is approximately cellular, it is the desired weak reflection (following Lemma 
\ref{L:cell}).
\end{proof}

\begin{rem}\label{soa}
{
\em
(1) In the proof of \ref{T:ref} we only need \ref{D:cell}(a).

(2) The proof mimics the construction of a weak reflection to $\Inj\ch$ (see e.g. 
\cite{AHRT}) -- just pushouts are replaced
by $\eps$-pushouts. By using pushouts, we would get $\bar{r}_K=\bar{m}_{0\lambda}:K\to \bar{K}$ with $\bar{K}$ approximately injective
w.r.t. $\ch$ but not a weak reflection to $\Inj_{\ap}\ch$. There is a morphism $t:\widehat K\to\bar{K}$ such that 
$tr_K=\bar{r}_K$ given by comparison morphisms from $\eps$-pushouts to pushouts.

(3) Observe that if $\ch'$ is the class of all the weak reflections $r_K:K\to\widehat{K}$, $K\in\ck$, then $\Inj_{\ap}\ch$
is the corresponding injectivity class:
$$
\Inj_{\ap}\ch=\Inj\ch'.
$$
}
\end{rem}

\begin{defi}
{
\em
The \textit{approximately cancellable closure} of  $\ch$  consists of all
morphisms $g:A\to B'$ such that for every $n>0$ there are morphisms $h:A\to B$ in $\ch$ and $f_n:B'\to B$ such that
$f_ng\sim_{\frac{1}{n}} h$.
}
\end{defi}

\begin{propo}\label{logic}
Suppose that domains of all morphisms of $\ch$ are approximately small w.r.t. $\ch$.
 Then a morphism  
in $\ck$ is an approximate injectivity consequence of $\ch$ iff  it  belongs to the approximately cancellable closure 
of $\cell_{\ap}(\ch)$. 
\end{propo}
\begin{proof}
Following \ref{L:cell} and \ref{L:cons}, every morphism from the approximately cancellative closure of $\cell_{\ap}(\ch)$ is an approximate injectivity consequence of $\ch$. Conversely, assume that a morphism $h\colon K\to L$ is an approximate injectivity consequence of $\ch$.
Let $r_K : K\to\widehat K$ be the approximately cellular weak reflection of \ref{T:ref}. Since $\widehat K$ is approximately injective, for each $n>0$ there is $f_n:L\to \widehat K$ such that $f_n \cdot g\sim{_\frac{1}{n}}r_K$. Hence $g$ belongs to the approximately cancellative closure of $\ch$.
\end{proof}

\begin{rem}
{
\em 
(1) Recall $\cell(\ch)$,  the class of cellular morphisms, from \ref{R:comp} (2). Let   $\overline{\cell}(\ch)$ be its cancellative closure given by morphisms $f$ with $hf\in\cell\ch$ for some $h$. Then
$$
\cell_{\ap}(\ch)\subseteq \overline{\cell}(\ch).
$$
Indeed, given an approximately cellular morphism $f:K\to L$, we take the corresponding cellular morphism $g:K\to M$ where we replace every $\varepsilon$-pushout by a pushout. Then $g=hf$ for some $h:L\to M$.

(2) If the domains of morphisms from $\ch$ are $\lambda$-small for some $\lambda$, then $\overline{\cell}(\ch)$ is the class of injectivity consequences of $\ch$. Thus every approximately cellular morphism w.r.t. $\ch$ is an injectivity consequence of $\ch$.

(3) Under the assumption of \ref{logic}, every strict approximate injectivity consequence of $\ch$ (see Definition~\ref{D:cons}) belongs 
to $\overline{\cell}_{\ap}(\ch)$.

Indeed, if $h:A\to B$ is a strict approximate injectivity consequence of $\ch$ and $r:A\to A^\ast$ is the approximately cellular weak reflection from \ref{T:ref}, then $r=gf$ for some $g$. 

On the other hand, one cannot expect that all morphisms from $\ch$ are strict approximate injectivity consequences of $\ch$.
}
\end{rem}

\section{Approximate injectivity in locally presentable categories}
We have recalled the concept of a locally presentable (ordinary) category in Remark \ref{R:Met}(2). Since we
work with $\Met$-enriched categories, we need the enriched concept. Let $\lambda$ be a regular cardinal. An object
$A$ is called \textit{$\lambda$-presentable in the enriched sense} if its hom-functor $\ck(A,-):\ck\to\Met$ preserves $\lambda$-directed colimits.
This means that (b) in Remark \ref{R:Met}(2) is strengthened to

(b') given $f',f'': A \to K_i$ then
$$d(k_if',k_if'') = \inf_{j\geq i}d(k_{i,j}f',k_{i,j}f'').$$

Following \cite{Ke}, a $\Met$-enriched $\ck$ is called \textit{localy $\lambda$-presentable in the enriched sense} if it has weighted colimits and a set
of $\lambda$-presentable objects in the enriched sense whose closure under $\lambda$-directed {colimits} is all of $\ck$.  Every such category has weighted limits.

For $\lambda$ uncountable, a $\Met$-enriched category $\ck$ with weighted colimits is locally $\lambda$-presentable in the enriched sense iff $\ck_0$ is locally $\lambda$-presentable and every $\lambda$-presentable object in $\ck_0$ is $\lambda$-presentable in the enriched sense.

\begin{propo}\label{coincide}
Let $\ck_0$ be locally $\lambda$-presentable. Then $\ck$ is locally $\lambda$-presentable 
in the enriched sense iff $\lambda$-presentable objects in $\ck_0$ are closed under $\eps$-coequalizers for every $\eps>0$.
\end{propo}

\begin{proof}
(1) Assume that $\lambda$-presentable objects in $\ck_0$ are closed under $\eps$-coequalizers for every $\eps>0$. Let $A$
 be $\lambda$-presentable 
in $\ck_0$. Given a $\lambda$-directed colimit $k_i:K_i \to K$ $(i \in I)$, it is our task to prove that for every parallel pair of morphisms $f',f'':A \to K_i$ with $d(k_if',k_if'')\leq\varepsilon$ there exists a connecting morphism $k_{i,j}: K_i \to K_j$  
for some $j\in I$ with $d(k_{ij}f',k_{i,j} f'') \leq \varepsilon.$ Since $\ck_0$ is locally $\lambda$-presentable, $K_i$ is a $\lambda$-directed colimit
$l_\alpha:L_\alpha\to K_i$, $\alpha\in J$ of $\lambda$-presentable objects $L_\alpha$. As $A$ is $\lambda$-presentable in $\ck_0$,
there are $\alpha\in J$ and $g',g'':A\to L_\alpha$ with $l_\alpha g'=f'$ and $l_\alpha g''=f''$. Let $c:L_\alpha\to C$ be 
an $\varepsilon$-coequalizer of $g'$ and $g''$. We have, due to 
$$
d(k_il_\alpha g,'k_il_\alpha g'')= d(k_if',k_if'')\leq\varepsilon,
$$
a factorization $k_il_\alpha=hc$ for some $h:C \to K$. Since $C$ is an $\eps$-coequalizer of $\lambda$-presentable objects 
in $\ck_0$, it is $\lambda$-presentable in $\ck_0$. Hence there is $i\leq j'\in I$ and $h':C\to K_{j'}$ such that $h=k_{j'}h'$.

We have
$$
k_{j'}k_{ij'}l_\alpha=k_il_\alpha=hc=k_{j'}h'c
$$
Since $C$ is  $\lambda$-presentable in $\ck_0$, there is $j'\leq j\in I$ such that
$$
k_{j'j}k_{ij'}l_\alpha=k_{j'j}h'c.
$$
Consequently, $k_{ij}l_\alpha=k_{j'j}k_{ij'}l_\alpha$ factorizes through $c$. Thus
$d(k_{ij}l_\alpha g',k_{ij}l_\alpha g'')\leq\eps$. Therefore $d({k_{ij}f',k_{ij}f'')}\leq\eps$.

(2) Conversely, assume that $\ck$ is locally $\lambda$-presentable in the enriched sense.
Let $c:B\to C$ be an $\eps$-coequalizer of $u,v:A\to B$ with $A$ and $B$ $\lambda$-presentable, then we are to prove that $C$ is $\lambda$-presentable. 
Let $k_i:K_i \to K$ $(i \in I)$ be a $\lambda$-directed colimit and $f:C\to K$. There are $i\in I$ and $f':B\to K_i$ such that $fc=k_if'$. Hence $k_if'u\sim_\eps k_if'v$ and, since $B$ is $\lambda$-presentable in $\ck$, there is $i\leq j\in I$ such that $k_{ij}f'u\sim_\eps k_{ij}f'v$. Thus $k_if'$ factorizes through $c$,
$k_{ij}f'=f''c$. Hence $f$ factorizes through $k_j$: we have $f=k_jf''$ because $c$ is epic by \ref{R:coeq} and 
	$$cf=k_if'=k_jf''c.$$
The essential uniqueness of this factorization is evident since $c$ is a coisometry by Assumptions \ref{A:enr}.

\end{proof}

\begin{coro}\label{coincide1}
Let $\ck_0$ be locally $\lambda$-presentable with $\lambda$-presentable objects closed under coisometric quotients.
Then $\ck$ is locally $\lambda$-presentable in the enriched sense.
\end{coro}
\begin{proof}
This follows from \ref{coincide} and \ref{A:enr}.
\end{proof}

\begin{exam}\label{assume1}
{
\em
The categories $\Met$, $\CMet$ and $\Ban$ are locally $\aleph_1$-presentable in the enriched sense.

Indeed, we have seen in Section 2 that the underlying ordinary categories are locally $\aleph_1$-presentable. Thus, we only need to observe that for every regular cardinal $\lambda$ the two concepts of $\lambda$-presentable object coincide. This follows from the above corollary: recall from Section 2 that if $\lambda$ is uncountable, then presentability in $\Met$ is just the cardinality being smaller than $\lambda$. Such spaces are clearly closed under coisometric quotients. And the only $\aleph_0$-presentable space is the empty one.  The situation with the other two categories is analogous.
}
\end{exam}

\medskip

The following theorem improves \cite{RT} 5.8. Recall the set-theory axiom
\begin{center}
\textit{Weak Vop\v{e}nka's Principle}
\end{center}
stating that no full embedding $\Ord^{\op} \hookrightarrow \Gra$ exists. Here $\Ord$ is the ordered category of all ordinals and $\Gra$ the category of graphs. This principle implies that measurable cardinals exist, thus, its negation is consistent with set theory by \cite {AR}, A.7. 
Weak Vop\v{e}nka's Principle is also consistent with set theory by \cite{AR}, 6.21 and A.12. Moreover, Weak Vop\v{e}nka's Principle follows from  \textit{Vop\v{e}nka's Principle} which states that a large discrete category cannot be fully embedded into $\Gra$
(\cite[Remark 6.21]{AR}).

\begin{theo}\label{WVP}
Assume Weak Vop\v{e}nka's Principle. If $\ck$ is locally $\lambda$-presentable in the enriched sense, then a class of objects is an approximate injectivity class iff it is closed under products and approximately split morphisms.
\end{theo}

For the proof see \cite{RT}, Theorem 5.8. The stronger assumption that Vop\v{e}nka's Principle holds made there was not fully used in that proof: by inspecting the proof one sees that all that was needed was the result that a full subcategory of a locally presentable category closed under products and split subobjects in weakly reflective. However, this was proved in \cite{AR93} under the following assumption:
\begin{center}
\textit{Semi-Weak Vop\v{e}nka's Principle}
\end{center}
stating that no class $L_i$ ($i\in \Ord$) of graphs fulfils $\Gra (L_i, L_j)= \emptyset$ iff $i<j$. Recently  Wilson proved that this principle is equivalent to Weak Vop\v{e}nka's Principle \cite{W}.

\section{Banach spaces}

We now turn to the category $\Ban$ and apply our results to prove that the Gurarii space is essentially unique. 

\begin{rem}\label{e-iso3}
{
	\em
	(1) We will show that, in $\Ban$,  $\eps$-isometries of Definition \ref{e-iso1} coincide with the usual concept of an $\eps$-isometry of Banach spaces 
	of norm $\leq 1$. Recall that a linear mapping $f:A\to B$ between Banach spaces is called an $\varepsilon$-\textit{isometry}  if
	it satisfies
	\[
	\tag{$\sharp$}(1-\varepsilon)\parallel x\parallel\leq\parallel fx\parallel\leq (1+\varepsilon)\parallel x\parallel 
	\] 
	for every $x\in A$ (see \cite{ASCGM}). It is evident that it suffices to take $x$ with $\parallel x\parallel= 1$ only.   
	
	(2) It should be a folklore that a linear mapping $f:A\to B$ between Banach spaces satisfies ($\sharp$) iff it satisfies
	\[
	\tag{$\ast$}|\parallel fx\parallel - \parallel x\parallel|\leq\varepsilon 
	\]
	for every $x\in A$, $\parallel x\parallel\leq 1$. We provide a short proof for the convenience of the reader.
	
	Let $f$ satisfy ($\sharp$) and choose $x$ with $\parallel x\parallel\leq 1$ such that $\parallel fx\parallel >\parallel x\parallel$. Then
	$$
	\parallel fx\parallel - \parallel x\parallel \leq (1+\varepsilon)\parallel x\parallel - \parallel x\parallel=\varepsilon\parallel x\parallel\leq\varepsilon.
	$$
	If $\parallel fx\parallel\leq\parallel x\parallel$ then $(1-\varepsilon)\parallel x\parallel\leq\parallel fx\parallel$ and thus
	$$
	\parallel x\parallel -\parallel fx\parallel \leq\varepsilon\parallel x\parallel\leq\varepsilon.
	$$
	Therefore $f$ satisfies ($\ast$).
	
	Conversely, let $f$ satisfy ($\ast$) and choose $x$ with $\parallel x\parallel = 1$. The inequality $\parallel fx\parallel\leq (1+\varepsilon)\parallel x\parallel$ is evident for $\parallel fx\parallel\leq\parallel x\parallel$. Assume that $\parallel x\parallel\leq\parallel fx\parallel$.
	Following ($\ast$), 
	\[
	\frac{|\parallel fx\parallel-\parallel x\parallel|}{\parallel x\parallel|}=|\frac{\parallel fx\parallel}{\parallel x\parallel}-1|=|\parallel f(\frac{x}{\parallel x\parallel})\parallel -\parallel\frac{x}{\parallel x\parallel}\parallel|\leq\varepsilon
	\]
	and thus
	$$\parallel fx\parallel\leq (1+\varepsilon)\parallel x\parallel.
	$$
	The other inequality $(1-\varepsilon)\parallel x\parallel\leq\parallel fx\parallel$ is evident for 
	$\parallel x\parallel\leq\parallel fx\parallel$. Assume that $\parallel x\parallel\geq \parallel fx\parallel$. 
	Following ($\ast$),
	$$
	\parallel x\parallel-\parallel fx\parallel \leq\varepsilon\parallel x\parallel
	$$
	and thus 
	$$
	(1-\varepsilon)\parallel x\parallel\leq \parallel fx\parallel.  
	$$
	Hence $f$ satisfies ($\sharp$).
}
\end{rem}

\begin{lemma}
A morphism $f:A\to B$ in $\Ban$ is an $\eps$-isometry iff it satisfies ($\sharp$).
\end{lemma}
\begin{proof}
Following \cite{G} 2.1, if a morphism $f: A\to B$ in $\Ban$ satisfies ($\sharp$) then $f$ has the property from \ref{e-iso2} and thus it is an $\eps$-isometry. Conversely, let $f$ be an $\eps$-isometry.
Then
$$
\pa x\pa - \pa fx\pa= |\pa x\pa - \pa fx\pa| = |\pa hx\pa - \pa gfx\pa|\leq \pa hx - gfx\pa\leq\eps
$$ 
Hence $f$ satisfies ($\sharp$). 
\end{proof}

\begin{rem}\label{e-iso3a}
{
	\em
	If a morphism $f:A\to B$ in $\Ban$ is an $\eps$-isometry between finite-dimensional Banach spaces, then the Banach space $C$ from \ref{e-iso1} can be taken finite-dimensional. The reason is that the $\eps$-pushout from \ref{e-iso2} is finite-dimensional (see \cite{G} 2.1).
}
\end{rem}

\begin{rem}\label{e-iso}
{
	\em
	(1) A mapping $f:A\to B$ between metric spaces (not necessarily non-expanding) is called an $\varepsilon$-\textit{isometry}
	in \cite{HU}) if 
	\[
	\tag{$\sharp\sharp$}|d(x,y)-d(fx,fy)|\leq\varepsilon
	\]
holds for all $x,y\in A$. Let us show that a morphism $f:A\to B$ in $\Met$ satisfies ($\sharp\sharp$) iff it is 
	a $\frac{\eps}{2}$-isometry in our sense. 
	
	Let $f:A\to B$ be a $\frac{\eps}{2}$-isometry, i.e., there are isometries $g:B\to C$ and $h:A\to C$ such that 
	$gf\sim_{\frac{\eps}{2}}h$. Then, for $x,y\in A$, we have
	$$
	d(x,y)=d(hx,hy)\leq d(hx,gfx) + d(gfx,gfy) + d(gfy,hy)=d(fx,fy)+\varepsilon.
	$$
	Hence $f$ satisfies ($\sharp\sharp$).
	
	Conversely, assume that $f$ satisfies ($\sharp\sharp$). Consider the $\frac{\varepsilon}{2}$-pushout used in \ref{e-iso2}. Since $u$ is an isometry (see \ref{stable}), we have
	$$
	d(\overline{f}x,\overline{f}y)\geq d(\overline{f}x,ufx) + d(ufx,ufy) + d(ufy,\overline{f}y) = \varepsilon + d(fx,fy) + \varepsilon = d(fx,fy) + 2\varepsilon.
	$$
	Since $f$ satisfies ($\sharp\sharp$), $d(fx,fy) + 2\varepsilon\geq d(x,y)$. Following the construction of $\eps$-pushouts
	in $\Met$ from \cite{RT} 2.3, we have $d(\overline{f}x,\overline{f}y)\geq d(x,y)$. Hence $\overline{f}$ is an isometry.
}
\end{rem}

\begin{rem}\label{class}
{
\em
The class of all isometries in $\Ban$ is stable under pushouts (see \cite{ASCGM} A.19) and hence under $\eps$-pushouts 
(see \ref{stable}). Since isometries in $\Ban$ are closed under transfinite composites, they are both cellularly closed and approximately cellularly closed.
}
\end{rem}

In \cite {RT} 6.5(1) it is claimed that finite-dimensional spaces are approximately $\aleph_0$-generated. We now prove this claim using Definition \ref{apgen}. This is, as remarked in \ref{apgen1}, weaker than that of \cite{RT} but the proof works for the stronger variant as well.

\begin{propo}\label{fin-dim}
Finite-dimensional Banach spaces are approximately $\aleph_0$-generated.
\end{propo}
\begin{proof}
Suppose $A$ is a finite-dimensional Banach space and let $(k_{ij} \colon K_i \to K_j)_{i\leq j\leq \mu}$ be a transfinite composite of isometries.

I. Let $f\colon A\to K_\mu$  be an isometry. 

(a) For every $\varepsilon>0$ we prove that there exists an $\varepsilon$-isometry $f'\colon A\to K_i$ (not necessarily of norm $\leq 1$, i.e.
satisfying $\sharp$ in \ref{e-iso3}) for some $i<\mu$ with $f\sim_\varepsilon k_{i\mu}f'$.
 
Let $e_1,\dots,e_n$ be a basis of $A$. Since any two norms on a finite-dimensional Banach space are equivalent, there is a number $r$ such that 
$$
\sum_{0<j\leq n} |a_j|\leq r\parallel\sum_{0<j\leq n} a_je_j\parallel.
$$
Let $\delta = \frac{\varepsilon}{r}$. There are elements $u_1,\dots,u_n\in K_i$, $i<\mu$, such that 
$$
|\parallel u_j\parallel - \parallel e_j\parallel|=|\parallel k_iu_j\parallel- \parallel fe_j\parallel|\leq\delta
$$
for $j=1,\dots, n$. Let $f':A\to K_i$ be the linear mapping such that $f'e_j=u_j$ for $j=1,\dots,n$. 
We have
$$
\parallel(k_{i,\mu}f'-f)(\sum_{0<j\leq n} a_je_j)\parallel\leq\sum_{0< j\leq n} |a_j|\parallel(k_{i,\mu}f'-f)(e_j)\parallel\leq\sum_{0< j\leq n}|a_j|\delta\leqslant\varepsilon\parallel\sum_{0<j\leq n} a_je_j\parallel.
$$
Hence $k_{i,\mu}f'\sim_\varepsilon f$.

Following \ref{e-iso3}, $f'$ is an $\varepsilon$-isometry because
$$
|\parallel f'\sum_{0<j\leq n}a_je_j\parallel - \parallel\sum_{0<j\leq n}a_je_j\parallel| = |\parallel k_if'\sum_{0<j\leq b}a_je_j|
\parallel - \parallel f\sum_{0<j\leq n}a_je_j\parallel|
$$
and this value is at most
$$
 \parallel k_if'\sum_{0<j\leq n}a_je_j - f\sum_{0<j\leq n}a_je_j\parallel =
\parallel(k_if'-f)\sum_{0<j\leq n} a_je_j\parallel \leq\varepsilon
$$
provided that $\parallel\sum_{0<j\leq n}a_je_j\parallel\leq 1$.
 
(b) For every $\varepsilon>0$ we next prove that there is $i<\mu$ and a morphism $f'':A\to K_i$ such that $\parallel k_if'' - f\parallel\leq\varepsilon$. 
 
Take $f'$ from the proof of (a) for $\varepsilon'= \frac{\varepsilon}{2}$. Let $\parallel\sum_{0<j\leq n}a_je_j\parallel=1$. Then
$$
|\parallel f'(\sum_{0<j\leq n}a_je_j)\parallel -1|\leq\varepsilon'.
$$
If $\parallel f'(\sum_{0<j\leq n}a_je_j)\parallel\geq 1$ then $\parallel f'(\sum_{0<j\leq n}a_je_j)\parallel\leq 1+\varepsilon'$.
If $\parallel f'(\sum_{0<j\leq n}a_je_j)\parallel \leq 1$ then, again, 
$\parallel f'(\sum_{0<j\leq n}a_je_j)\parallel\leq 1+\varepsilon'$. We have proved that $\parallel f'\parallel\leq 1+\varepsilon'$.
Hence $f''= \frac{1}{1+\varepsilon'} f'$ is a morphism in $\Ban$.

For $a=\sum_{0<j\leq n}a_je_j$ we have
$$
\parallel  f'a - f''a\parallel = \frac{1}{1+\varepsilon'} \parallel (1+\varepsilon')f'a - f'a)\parallel = \frac{\varepsilon'}{1+\varepsilon'}\parallel f'a\parallel\leq\frac{\varepsilon'}{1+\varepsilon'}(1+\varepsilon')\parallel a\parallel=\varepsilon'\parallel a\parallel.
$$
Hence $\parallel f'-f''\parallel\leq\varepsilon'$ and thus $\parallel k_if'-k_if''\parallel\leq\varepsilon'$. Since
$\parallel k_if' - f\parallel\leq\varepsilon'$, we have $\parallel k_if'' - f\parallel\leq\varepsilon$.

 
II. Every morphism $f:A\to K_\mu$ has a decomposition 
$$
A\to f[A]\to K_\mu
$$
where $f_1:A\to f[A]$ is an  epimorphism and $f_2: f[A]\to K_\mu$ is an isometry. Since $f[A]$ is a finite-dimensional Banach space
and $f_2$ is an isometry, following I. we have a morphism $f'_2:f(A)\to K_i$ such that $\pa k_if'_2 - f_2\pa\leq\eps$. Now, for
$f'=f'_2f_1$ we have $\pa k_if' - f\pa\leq\eps$.
\end{proof}

\begin{rem}\label{apgen2}
{
\em
(1) An example of an approximately $\aleph_0$-generated Banach space which is not
finite-dimensional is given in \cite[Example 6.4]{DR}.

Following \cite{RT} 6.5(2), every approximately $\aleph_0$-generated Banach space $A$ admits for every $\varepsilon>0$ 
an $\varepsilon$-split morphism $u:A\to B$ to a finite-dimensional Banach space $B$. This means that there exists $r:B\to A$ such
that $ru\sim_{\eps}\id_A$. Moreover, $r$ can be taken as an isometry.

Conversely, every Banach space $A$ with morphisms $r$ and $u$ as above is approxi\-ma\-tely $\aleph_0$-generated. Indeed, let 
$(k_{ij}:K_i\to K_j)_{i\leq j\leq\mu}$ a transfinite composite of isometries and let $f:A\to K_\mu$ and $\eps>0$ be given. There is an $\frac{\varepsilon}{2}$-split morphism $u:A\to B$ with $B$ finite-dimensional. Since $B$ is approximately $\aleph_0$-generated, there is  $f':B\to K_i$ for some $i<\mu$ such that 
$k_if'\sim_{\frac{\varepsilon}{2}}fr$ (where $r$ $\eps$-splits $u$). Hence $k_if'u\sim_{\frac{\varepsilon}{2}}fru$ and 
$fru\sim_{\frac{\varepsilon}{2}}f'$. Therefore $k_if'\sim_\varepsilon f$.

(2) Every approximately $\aleph_0$-generated Banach space $A$ is separable. Indeed, for every $n>0$ there is a $\frac{1}{n}$-split morphism $u_n:A\to B_n$ to a finite-dimensional Banach space $B_n$. Let $X_n$ be a countable dense set in $B_n$. Then for
$r_n:B_n\to A$ which $\frac{1}{n}$-splits $u_n$ we conclude that $\bigcup\limits_{n>0}r_n(X_n)$ is a countable dense set in $A$.



}
\end{rem}

\begin{rem}
{
\em
$\Ban$ is \textit{isometry-locally finitely generated} as a $\CMet$-enriched category (cf. \cite{AR1}). It suffices to combine \ref{fin-dim} and \ref{small:iso}(3) with the fact that every Banach space is a directed colimit of finite-dimensional Banach spaces and isometries.
}
\end{rem}

\begin{rem}\label{gur}
{
\em 
(1) \textit{Approximately $\aleph_0$-saturated} objects were defined in \cite{RT} 6.6. as objects approximately injective w.r.t.
isometries between approximately $\aleph_0$-generated objects in the category $\ck_{\iso}$ of $\ck$-objects and isometries.
This means that in \ref{ain} $\ch$ consists of isometries between approximately $\aleph_0$-generated objects and $f,f'$ are isometries.

(2) An approximately $\aleph_0$-saturated object in $\Ban$ is approximately injective w.r.t. isometries between approximately $\aleph_0$-generated Banach spaces. It suffices to take an arbitrary $f:A\to X$ from \ref{ain}, factorize it as $f=f_2f_1$ where $f_2:A_0\to X$
is an isometry and $A_0$ is finite-dimensional and take the following pushout 
$$
\xymatrix@=3pc{\ar@{} [dr]|{}
A \ar [r]^{h} \ar[d]_{f_1} &  A'\ar[d]^{\overline{f}_1}\\
A_0\ar[r]_{\overline{h}} &A'_0
}
$$
Then for $f_2:A_0\to X$, we take $f'_2:A'_0\to X$ such that $f'_2\overline{h}\sim_{\eps}f_2$. We conclude 
$f'_2\overline{f}_1h\sim_{\eps}f$.

(3) Banach spaces of almost universal disposition for finite dimensional Banach spaces can be defined as Banach spaces approximately injective w.r.t.
isometries between finite-dimensional Banach spaces in the category $\Ban_{\iso}$. The definition is formulated differently in \cite{ASCGM}, but its equivalence with the above condition is shown in \cite{K}.

Every approximately $\aleph_0$-saturated Banach space is clearly of almost universal disposition for finite dimensional Banach spaces. 

(4) Gurarii constructed a separable Banach space of almost universal disposition for finite dimensional Banach spaces, \cite{Gu}. This space is unique up to isomorphism (see \cite{KS}) and is called the \textit{Gurarii space}. On the other hand, there exists no separable Banach space injective w.r.t. isometries between finite-dimensional Banach spaces in $\Ban$ (see \cite{ASCGM}, 3.10).  

In \cite{RT} 6.7, a separable $\aleph_0$-saturated Banach space was obtained as $\bar{r}_{K_0}=\bar{m}_{0\omega}:K_0\to\bar{K}_0$ (from \ref{soa}) where $K_0=0$ is the null space. It could also be obtained as $r_{K_0}=m_{0\omega}:K_0\to\widehat{K}_0$ from \ref{T:ref} and \ref{class}. This space is the Gurarii space and thus the Gurarii
space is approximately $\aleph_0$-saturated.



}
\end{rem}

We are going to prove the uniqueness of a Gurarii space based on the results above, but using the approximate back-and-forth method of \cite{KS}. 

\begin{lemma}\label{e-iso4}
Let $K$ be an approximately $\aleph_0$-saturated Banach space. Given an $\varepsilon$-isometry $h:A\to A'$ in $\Ban$ between 
finite-dimensional Banach spaces and an isometry $f:A\to K$, then for every $\delta>0$ there is an isometry  
$f':A'\to K$ in $\Ban$ such that $f'h\sim_{\eps+\delta} f$.
\end{lemma}
\begin{proof}
Following \ref{e-iso3a}, there are isometries $u:A'\to B$ and $v:A\to B$ where $B$ is finite-dimensional such that 
$uh\sim_{\eps}v$. Given $\delta>0$, there is an isometry $f'':C\to K$ such that $f''v\sim_\delta f$. 
Let $f'=f''u$. Then $f'h\sim_{\varepsilon + \delta}f$.
\end{proof}
 
\begin{propo}
Any two separable approximately $\aleph_0$-saturated Banach spaces are isomorphic.
\end{propo}
\begin{proof}
Let $K$ and $L$ be separable approximately $\aleph_0$-saturated Banach spaces. Both $K$ and $L$ are colimits of an $\omega$-chain of 
finite-dimensional Banach spaces and linear isometries: $K = \colim\limits_{i<\omega} K_i$ and $L = \colim\limits_{j<\omega} L_j$ 
where $K_0=L_0=0$ and with the connecting morphisms $k_{i_1i_2}:K_{i_1}\to K_{i_2}$ and $l_{j_1j_2}:L_{j_1}\to L_{j_2}$.
We will produce increasing sequences $i_1<\dots i_n<\dots$ and $j_0=0 < j_1<\dots j_n<\dots$, together with $\frac{1}{n}$-isometries $f_n:K_{i_n}\to L_{j_n}$, $n>0$ in $\Ban$ and $\frac{1}{n+1}$-isometries $g_n:L_{j_n}\to K_{i_n+1}$, $n\geq 0$ in $\Ban$ such that for every $n>0$ we have 
\begin{equation}
\parallel g_nf_n -k_{i_n,i_{n+1}}\parallel\leq\frac{2}{n+1}\tag{$\ast$}
\end{equation}
and
\begin{equation}
\parallel f_{n+1}g_n -l_{j_n,j_{n+1}}\parallel\leq\frac{2}{n+1}\tag{$\ast\ast$}.
\end{equation}
We proceed by induction, $g_0=\id_0: 0\to 0$. Assume that we have $g_n$. Following \ref{e-iso4}, there is an isometry $t:K_{i_n}\to L$ such that 
$$
\parallel t g_n - l_{j_n}\parallel\leq\frac{1}{n+1}.
$$
There is an index $j_{n+1}>j_n$ for which we have a morphism $f_{n+1}:K_{i_{n+1}}\to L_{j_{n+1}}$ such that 
$\pa l_{j_{n+1}}f_{n+1}-t\pa\leq{\frac{1}{n+1}}$. Following \ref{e-iso1}, $f_{n+1}$ is an $\frac{1}{n+1}$-isometry. Moreover,
$$
\pa l_{j_{n+1}}f_{n+1}g_n - l_{j_{n+1}}l_{j_nj_{n_1}}\pa=\pa l_{j_{n+1}}f_{n+1}g_n -tg_n +tg_n- l_{j_{n+1}}l_{j_nj_{n_1}}\pa
\leq \frac{1}{n+1} + \frac{1}{n+1}.
$$
Hence we have $(\ast\ast)$.
 
Next, assume that we have $f_n$. Following \ref{e-iso4}, there is an $\frac{1}{n+1}$-isometry $t:L_{j_n}\to K$ such that
$$
\parallel t f_n - k_{i_n}\parallel\leq\frac{1}{n+1}.
$$
There is an index $i_{n+1}>i_n$ for which we have a morphism $g_n:L_{j_n}\to K_{i_{n+1}}$ such that $\pa k_{i_{n+1}}g_n-t\pa\leq{\frac{1}{n+1}}$. Following \ref{e-iso1}, $g_n$ is an $\frac{1}{n+1}$-isometry. Moreover,
$$
\pa k_{i_{n+1}}g_nf_n - k_{i_{n+1}}k_{i_ni_{n_1}}\pa=\pa k_{i_{n+1}}g_nf_n -tf_n +tf_n- k_{i_{n+1}}k_{i_ni_{n_1}}\pa
\leq \frac{1}{n+1} + \frac{1}{n+1}.
$$
Hence we have $(\ast)$.
 
Then the morphisms $f=\lim_nf_n:K\to L$ and $g=\lim_ng_n:L\to K$ are mutually inverse. In more detail, every $x\in K$ is a limit $x=\lim_n x_n$
where $x_n\in K_{i_n}$. Then $fx = \lim_n f_nx_n$. Similarly, we define $g$. 

Hence $K\cong L$.
\end{proof}

\end{document}